\documentclass[11pt,letterpaper]{amsart}

\usepackage[T1]{fontenc}
\usepackage[utf8]{inputenc}
\usepackage{lmodern}
\usepackage{amsmath,amssymb,amsfonts,mathrsfs}
\usepackage{microtype}
\usepackage[hidelinks]{hyperref}
\hypersetup{pdftitle={Stability of the Riemannian positive mass theorem in all dimensions},
pdfauthor={Gaoming Wang and Yiyue Zhang}}

\allowdisplaybreaks
\numberwithin{equation}{section}

\newtheorem{theorem}{Theorem}[section]
\newtheorem{proposition}[theorem]{Proposition}
\newtheorem{lemma}[theorem]{Lemma}
\newtheorem{corollary}[theorem]{Corollary}
\theoremstyle{remark}
\newtheorem{remark}[theorem]{Remark}

\title[Stability of the positive mass theorem]
{Stability of the Riemannian positive mass theorem in all dimensions}
\author{Gaoming Wang}
\address{
Beijing Institute of Mathematical Sciences and Applications,
Beijing 101408, China
}
\email{wanggaoming@bimsa.cn}

\author{Yiyue Zhang}
\email{zhangyiyue@bimsa.cn}


\begin{document}
\begin{abstract}
We prove stability of the Riemannian positive mass theorem in all dimensions, extending the Dong–Song stability theorem. For a sequence of complete asymptotically flat manifolds with nonnegative scalar curvature and ADM masses tending to zero, excising domains whose boundary areas tend to zero yields exterior regions converging to Euclidean space in the pointed measured Gromov–Hausdorff topology. The proof constructs global coordinates from minimal graphs and controls their Hessians using scalar solutions of the conformal Laplace equation on the associated graph metrics. 
\end{abstract}
\maketitle

\section{Introduction and main results}\label{sec:introduction}

The positive mass theorem asserts that a complete asymptotically flat
manifold with nonnegative scalar curvature has nonnegative ADM mass,
and that zero mass characterizes Euclidean space. Schoen and Yau gave the first proof of the theorem, and Witten later
proved it for spin manifolds in every dimension
\cite{SchoenYau1979,Witten1981,ParkerTaubes1982}. 
The recent dimension descent argument of Brendle and Wang
\cite{BrendleWang2026}, and its extension by Khuri, Wang, and Wang
\cite{KhuriWangWang2026}, give the theorem in every dimension.
We prove the corresponding stability statement.

The stability question asks whether small mass forces a suitable exterior
region to be close to Euclidean space. Deep wells prevent convergence
of the whole manifold in the Gromov--Hausdorff topology. Huisken and
Ilmanen therefore proposed removing a region whose boundary area is
small \cite[Section~9]{HuiskenIlmanen2001}. The intrinsic flat distance
of Sormani and Wenger provides an alternative framework for studying
stability in the presence of deep wells
\cite{SormaniWenger2011}. Stability results in this topology include
rotationally symmetric, graphical, and geometrostatic manifolds
\cite{LeeSormani2014,HuangLee2015,HuangLeeSormani2017,
HuangLeePerales2022,SormaniStavrovAllen2019}. See also
\cite{Sormani2023}.

In dimension three, the harmonic function mass inequality of Bray,
Kazaras, Khuri, and Stern \cite{BrayKazarasKhuriStern2022} led to a
coordinate approach to stability. Kazaras, Khuri, and Lee proved
pointed Gromov--Hausdorff stability under additional asymptotic,
curvature, and topological assumptions \cite{KazarasKhuriLee2024}.
Dong and Song then proved the full excision statement, with pointed
measured Gromov--Hausdorff convergence and arbitrary base points
\cite{DongSong2025}. Related results include stability in the K\"ahler
setting \cite{Klemmensen2025} and higher dimensional spin stability
under a Ricci lower bound \cite{BrydenXie2026}. Related methods have also
been applied to stability results for Llarull's theorem
\cite{AllenBrydenKazaras2025,HirschZhang2024} and spacetime Penrose
inequalities \cite{AllenBrydenKazarasKhuri2025}.

Our proof combines global minimal graph coordinates with scalar
solutions of the conformal Laplace equation. A conformal change of
each graph metric produces nonnegative scalar curvature containing
the squared Hessian of its coordinate function. The mass of this
metric is an explicit dimensional multiple of the original mass.
Applying the qualitative positive mass inequality to a second
conformal change bounds an energy involving this scalar curvature.
The resulting estimates control the coordinate defect on small
sublevel sets. An excision argument then gives global coordinates
with small metric distortion and small boundary area.

An \emph{exterior region} is a connected, closed submanifold of dimension
\(n\), with smooth compact boundary, containing the distinguished end
outside a compact set. Empty boundary is allowed. A \emph{Euclidean
exterior region} is defined in the same way and contains
\(\mathbb R^n\setminus B_R\) for some \(R\).
We write \(\widehat d_g\) for the intrinsic length metric. Pointed
measured Gromov--Hausdorff convergence uses the restricted Riemannian
measures, without normalization. The asymptotic decay and mass
normalization are specified in Section~\ref{sec:preliminaries}.

\begin{theorem}\label{thm:main}
Fix \(n\geq3\). Let \((M_i^n,g_i)\) be smooth, connected, complete,
oriented manifolds without boundary and with finitely many ends, all
asymptotically flat. Suppose that
\[
 R_{g_i}\geq0,\qquad R_{g_i}\in L^1(M_i).
\]
Choose a distinguished end of each \(M_i\), and let \(m_i\) be its
ADM mass. If \(m_i\to0\), there are smooth connected exterior regions
\(\widehat E_i\subset M_i\) containing the distinguished ends such that
\[
 \mathcal H_{g_i}^{n-1}(\partial\widehat E_i)\longrightarrow0.
\]
For every choice of base points \(x_i\in\widehat E_i\),
\[
 (\widehat E_i,\widehat d_{g_i},x_i,dV_{g_i})
 \xrightarrow{\mathrm{pmGH}}
 (\mathbb R^n,d_{\mathrm{Eucl}},0,dx).
\]
\end{theorem}

The following theorem relates the boundary area of the exterior region
to its metric distortion.

\begin{theorem}\label{thm:quantitative}
Let \((M^n,g)\), \(n\geq3\), satisfy the geometric hypotheses of
Theorem~\ref{thm:main}, with \(R_g\geq0\) and \(R_g\in L^1(M)\).
Suppose that the distinguished end has harmonic asymptotics, and put
\(m:=m(g)\).  If \(m>0\), there are constants
\(\epsilon_0(n)>0\) and \(C(n)<\infty\) such that, for every
\(0<\epsilon<\epsilon_0(n)\), there are exterior regions
\(E_\epsilon\subset M\) and \(\Omega_\epsilon\subset\mathbb R^n\)
and an orientation-preserving diffeomorphism
\[
 \Phi_\epsilon:E_\epsilon\longrightarrow\Omega_\epsilon
\]
satisfying
\begin{equation}
 \mathcal H_g^{n-1}(\partial E_\epsilon)
 \leq C(n)\left(\frac m\epsilon\right)^{\frac{n-1}{n-2}},
 \qquad
 \sup_{\Omega_\epsilon}
 |(\Phi_\epsilon^{-1})^*g-\delta|_\delta
 \leq C(n)\sqrt\epsilon.
 \label{eq:quantitative-conclusions}
\end{equation}
At the distinguished end,
\(\Phi_\epsilon(x)=x+o(r)\) and
\(d\Phi_\epsilon=\operatorname{Id}+o(1)\).
\end{theorem}

For a sequence with harmonic asymptotics and positive masses, choose
\(\epsilon_i:=\sqrt{m_i}\) for all sufficiently large \(i\).
The boundary and metric errors in \eqref{eq:quantitative-conclusions}
are then
\begin{equation}
 O\!\left(m_i^{\frac{n-1}{2(n-2)}}\right)
 \quad\text{and}\quad O(m_i^{1/4}),
 \label{eq:1.3}
\end{equation}
respectively. More generally, both errors tend to zero whenever
\(\epsilon_i\to0\) and \(m_i/\epsilon_i\to0\).
The proof uses only the inequality part of the positive
mass theorem. 

We briefly outline the key ideas of the proof. We construct global
minimal graph functions $u_1,\ldots,u_n$ whose normalized components
define a map $\Phi:=\lambda^{-1}(u_1,\ldots,u_n)$ asymptotic to the
identity on the distinguished end. A conformal change of each graph
metric produces nonnegative scalar curvature containing a positive
quadratic term in $\nabla_g^2u_a$. Further conformal deformations
and the positive mass theorem yield weighted $L^2$ bounds for these
Hessians in terms of the original ADM mass. Adapting the level
selection argument of Dong and Song, we obtain an exterior region
with small boundary area on which $\Phi$ is a diffeomorphism with
small metric distortion. Finally, we apply Dong's higher dimensional
extension of their Euclidean excision theorem to obtain pointed
measured Gromov--Hausdorff convergence to $\mathbb R^n$ with respect
to the intrinsic metrics.

\medskip
\noindent\textbf{Organization of the paper.}
Section~2 sets notation for weighted spaces and asymptotic geometry
and recalls the positive mass theorem. Section~3 constructs the
global minimal graph functions and establishes their asymptotics
(Theorem~3.1). Sections~4 and~5 derive the scalar curvature identity
and the ADM mass formula for the conformal graph metrics, respectively.
Section~6 constructs the scalar conformal potential and proves the
mass--energy estimate (Proposition~6.3) and its localized versions.
Section~7 establishes the energy bound for the coordinate defect and
treats the zero mass case. Section~8 selects a regular level set of
small area and extracts the exterior coordinates, completing the
proof of Theorem~1.2. Finally, Section~9 proves convergence in the
intrinsic metrics and removes the harmonic asymptotics assumption
to complete the proof of Theorem~1.1.

\medskip
\noindent\textbf{AI usage.} The authors made substantial use of ChatGPT in developing and editing this manuscript, including suggestions for the main
arguments. The authors carefully checked all
arguments and rewrote the proofs.

\medskip
\noindent\textbf{Acknowledgements.} We thank Sven Hirsch, Demetre Kazaras and Marcus Khuri for
many helpful discussions and suggestions. YZ is partially supported by NSFC
Grant No.\ 12501070. Both authors are partially supported by the startup fund from BIMSA.

\section{Preliminaries}\label{sec:preliminaries}

\subsection{Weighted spaces and asymptotic geometry}
\label{subsec:asymptotic-geometry}

The following weighted norms describe the decay of the metric, graph
functions, and conformal potentials, including their derivatives on
rescaled annuli.

Fix \(0<\alpha<1\) and an inner radius \(R_0\) for the exterior
coordinate chart.  For any weight \(s\in\mathbb R\), define the weighted
H\"older norm of a function or tensor \(z\) by
\begin{equation}
 \|z\|_{C^{k,\alpha}_{s}}
 :=\sup_{R\geq R_0}R^{-s}
 \bigl\|z(R\,\cdot)\bigr\|_{C^{k,\alpha}(\{1\leq|y|\leq2\})}.
 \label{eq:weighted-holder-norm}
\end{equation}
For tensors, the norm is applied componentwise to their coordinate
components.  Thus \(z=O_{k,\alpha}(r^s)\) means
\(z\in C^{k,\alpha}_s\), while \(z=o_{k,\alpha}(r^s)\) means that the
corresponding rescaled annular norm tends to zero.

An end \(E\) of an \(n\)-manifold is \emph{asymptotically flat} (AF) if it
is diffeomorphic to \(\mathbb R^n\setminus B_{R_0}\) and, in the
corresponding coordinates,
\begin{equation}
 g_{ij}-\delta_{ij}\in C^{2,\alpha}_{-\tau},
 \qquad \tau>\frac{n-2}{2}.
 \label{eq:af-definition}
\end{equation}
We additionally assume that \(R_g\in L^1(M)\).  Denote
\(\omega_{n-1}:=\mathcal H^{n-1}(S^{n-1})\), and write
\(S_r:=\{|x|=r\}\), \(\nu:=x/r\), and \(dA_\delta\) for, respectively,
the Euclidean coordinate sphere, its outward Euclidean unit normal, and
its Euclidean area measure.  The ADM mass of this end is normalized by
\begin{equation}
 m_E(g):=\frac{1}{2(n-1)\omega_{n-1}}
 \lim_{r\to\infty}\int_{S_r}
 (\partial_jg_{ij}-\partial_ig_{jj})\nu^i\,dA_\delta.
 \label{eq:adm-mass}
\end{equation}
For the distinguished end we abbreviate \(m_E(g)\) to \(m(g)\).

An AF end is \emph{harmonically flat} if, outside a compact set,
\[
 g=\phi^{4/(n-2)}\delta,
 \qquad \Delta_\delta\phi=0,
 \qquad \phi>0,
 \qquad \phi\longrightarrow1.
\]
More generally, it has \emph{harmonic asymptotics} if, after changing
the AF coordinates if necessary,
\begin{equation}
 g=\phi^{4/(n-2)}\delta+O_{2,\alpha}(r^{1-n}),
 \qquad
 \phi=1+\frac{m_E(g)}2 r^{2-n}+O_{2,\alpha}(r^{1-n}),
 \label{eq:harmonic-asymptotics}
\end{equation}
where \(\phi\) is positive and Euclidean harmonic near infinity. The
coefficient is written using the normalization \eqref{eq:adm-mass}.

Unless stated otherwise, manifolds are smooth, connected, complete,
oriented, without boundary, and have finitely many AF ends, one of which is
distinguished.  All AF coordinate charts are chosen to preserve the given
orientation.  For the exterior regions defined before
Theorem~\ref{thm:main}, no condition is imposed on which auxiliary ends
they contain, and PDE statements concern their interiors.

We use \(\Delta_g:=\operatorname{div}_g\nabla\). For matrices, \(|\cdot|\)
is the Euclidean Hilbert--Schmidt norm and \(I\) is the identity matrix.
Weighted norms on \(M\) combine the end norms above with the usual
norms on a fixed compact core.

\subsection{The Riemannian positive mass theorem}
The Riemannian positive mass theorem in all dimensions was established in
\cite{BrendleWang2026}. A more general version allowing singularities
was proved in \cite{KhuriWangWang2026}.

\begin{theorem}[\cite{BrendleWang2026,KhuriWangWang2026}]\label{thm:positive-mass}
Let \((N^n,\gamma)\), \(n\geq3\), be a smooth, connected, complete
manifold without boundary and with finitely many asymptotically flat
ends in the sense of \eqref{eq:af-definition}. If
\(R_\gamma\geq0\) and \(R_\gamma\in L^1(N)\), the ADM mass of
every end is nonnegative.
\end{theorem}

\section{Global affine minimal graphs}\label{sec:global-graphs}

We construct the functions that will form the coordinate map \(\Phi\).
Throughout this section, the distinguished end \(E_0\) has harmonic
asymptotics \eqref{eq:harmonic-asymptotics}, while the auxiliary ends need only
satisfy \eqref{eq:af-definition}. Put \(m:=m(g)\). On \(E_0\),
\begin{equation}\label{eq:graph-harmonic}
 g=(1+f_0)\delta+O_{2,\alpha}(r^{1-n}),\qquad
 f_0:=\frac{2m}{n-2}r^{2-n}.
\end{equation}
Coordinate derivatives are denoted by \(D\), and covariant derivatives
by \(\nabla_g\). For a smooth function \(v\), set
\begin{equation}\label{eq:graph-operator}
 W_v:=(1+|dv|_g^2)^{1/2},\qquad
 \mathcal M_g(v):=\operatorname{div}_g\frac{\nabla_g v}{W_v}.
\end{equation}
Thus \(\mathcal M_g(v)=0\) is the minimal graph equation in
\((M\times\mathbb R,g+dt^2)\). A supersolution satisfies
\(\mathcal M_g(v)\leq0\), and a subsolution satisfies the reverse
inequality. The equivalent nondivergence operator is
\begin{equation}\label{eq:graph-nondivergence}
 \mathcal P_g(v):=W_v\mathcal M_g(v)
 =\left(g^{ij}-\frac{v^iv^j}{1+|dv|_g^2}\right)(\nabla_g^2v)_{ij},
 \qquad v^i:=g^{ij}\partial_jv.
\end{equation}
For a slope \(p\in\mathbb R^n\), define
\begin{equation}\label{eq:graph-parameters}
 \eta_0:=(1+|p|^2)^{-1/2},\qquad
 A:=I+p\otimes p,\qquad
 \mathfrak c_p:=\frac{n-2+\eta_0^2|p|^2}{2}.
\end{equation}

\begin{theorem}\label{thm:affine-graph}
There is \(\lambda_0(n)>0\) with the following property.
For every manifold as above and every
\(0<|p|\leq\lambda_0(n)\), there is a smooth solution \(u_p\) of
\(\mathcal M_g(u_p)=0\) on \(M\) such that
\begin{equation}\label{eq:graph-main-asymptotics}
 u_p=p\cdot x+w_0+O_{3,\alpha}(r^{2-n+1/4})\quad\text{on }E_0,
\end{equation}
where \(w_0\) is homogeneous of degree \(3-n\) and solves
\begin{equation}\label{eq:graph-w0-equation}
 \operatorname{div}_\delta(A^{-1}Dw_0)=-\mathfrak c_p\,p\cdot Df_0.
\end{equation}
In dimension three, choose the odd solution of \eqref{eq:graph-w0-equation}.
On each auxiliary end there is \(b_E\in\mathbb R\) such that
\begin{equation}\label{eq:graph-aux-asymptotics}
 u_p=b_Er^{2-n}+O_{3,\alpha}(r^{2-n-1/8}).
\end{equation}
The solution is unique among smooth solutions for which
\(u_p-p\cdot x-w_0\to0\) on \(E_0\) and \(u_p\to0\) on every auxiliary end.
Its gradient is globally bounded, and its graph metric is complete.
\end{theorem}

The upper bound for the slope depends only on the dimension, whereas the
starting radius and estimates on compact subsets may depend on the manifold.

\subsection{Dirichlet solvability and local estimates}

A smooth bounded domain is strictly mean convex if its boundary has
positive outward mean curvature. Equivalently, its inward distance
function \(d\) satisfies \(\Delta_gd\leq-h_0<0\), for some
\(h_0>0\), on a sufficiently thin closed interior collar.
Large coordinate spheres bounding a truncated AF manifold have this property:
their outward mean curvature is \((n-1)/r+o(r^{-1})\).

We recall the standard Dirichlet solvability result for minimal graphs on
bounded strictly mean convex domains. See Jenkins--Serrin
\cite{JenkinsSerrin1968} and Serrin \cite{Serrin1969} for the classical
Euclidean theory, and Spruck \cite[Sections~3--4]{Spruck2007} for the
Riemannian estimates and continuity argument.

\begin{lemma}\label{lem:graph-dirichlet}
Let \(\Omega\Subset M\) have smooth strictly mean convex boundary.
For every \(\varphi\in C^\infty(\partial\Omega)\), there is a unique smooth
solution of
\[
 \mathcal M_g(u)=0\quad\text{in }\Omega,\qquad
 u=\varphi\quad\text{on }\partial\Omega.
\]
\end{lemma}

For the exhaustion argument below, we also need local compactness of these
solutions. Spruck's interior gradient estimate
\cite[Theorem~1.1]{Spruck2007}, combined with standard elliptic estimates
\cite[Chapters~6 and~13--15]{GilbargTrudinger2001}, shows that uniform local
height bounds yield uniform derivative bounds on smaller compact subsets.

\subsection{An approximate affine solution}

\begin{lemma}\label{lem:graph-model}
For \(|p|\leq1\), there is a smooth homogeneous function \(w_0\) of degree \(3-n\)
on \(\mathbb R^n\setminus\{0\}\) satisfying \eqref{eq:graph-w0-equation}.
It can be chosen odd, and
\begin{equation}\label{eq:graph-w0-bounds}
 |D^jw_0|\leq C_{n,j}|m||p|r^{3-n-j}.
\end{equation}
The function \(v_0:=p\cdot x+w_0\) satisfies on the distinguished end
\begin{equation}\label{eq:graph-model-error}
 \mathcal M_g(v_0)=O_{1,\alpha,M}(|p|r^{-n}).
\end{equation}
After increasing the starting radius, uniformly for \(|p|\leq1\),
\begin{equation}\label{eq:graph-model-smallness}
 |v_0|\leq D_n|p|r,\qquad
 |dv_0|_g+r|\nabla_g^2v_0|_g\leq D_n|p|.
\end{equation}
The subscript \(M\) in \eqref{eq:graph-model-error} allows dependence on the fixed end,
but not on \(p\) in this interval.
\end{lemma}

\begin{proof}
The change of variables \(y:=A^{1/2}x\) transforms
\(\operatorname{div}_\delta(A^{-1}D)\) into \(\Delta_y\).
The right-hand side of \eqref{eq:graph-w0-equation} is
\(-\mathfrak c_p\,p\cdot Df_0=2m\mathfrak c_p(p\cdot x)|x|^{-n}\), which is homogeneous of
degree \(1-n\) and odd in \(x\). Both properties are preserved under this
linear change of variables.
Setting \(\psi_p(\theta):=w_0(A^{-1/2}\theta)\), homogeneity reduces the equation to
\begin{equation}\label{eq:graph-sphere}
 (\Delta_{S^{n-1}}+3-n)\psi_p=F_p,\qquad
 F_p(\theta):=2m\mathfrak c_p(p\cdot A^{-1/2}\theta)
 |A^{-1/2}\theta|^{-n}.
\end{equation}
Here \(\theta\in S^{n-1}\) is a unit vector in the \(y\) coordinates.
For \(n\geq4\) this operator is invertible.
For \(n=3\), \(F_p\) is odd and hence has mean zero, so we choose the unique
odd solution.
For \(|p|\leq1\), the eigenvalues of \(A\) lie in \([1,2]\), and the smooth norms of
\(F_p\) are bounded by dimensional constants times \(|m||p|\).
Spherical elliptic estimates prove \eqref{eq:graph-w0-bounds}.

We compute the error explicitly.
In the asymptotic coordinates, set
\begin{equation}\label{eq:graph-flux-density}
 J_g(v):=\sqrt{\det g}\frac{g^{-1}Dv}{\sqrt{1+|Dv|_g^2}},
 \qquad
 \mathcal M_g(v)=(\det g)^{-1/2}\partial_iJ_g^i(v).
\end{equation}
From \eqref{eq:graph-harmonic}, Taylor expansion gives
\begin{align*}
 \sqrt{\det g}\,g^{-1}
 &=I+\frac{n-2}{2}f_0 I+O_{2,\alpha,M}(r^{1-n}),\\
 (1+|p+Dw_0|_g^2)^{-1/2}
 &=\eta_0\left(1-\eta_0^2p\cdot Dw_0
       +\frac12\eta_0^2|p|^2f_0\right)
       +O_{2,\alpha,M}(|p|^2r^{1-n}),
\end{align*}
where the last line uses \eqref{eq:graph-w0-bounds} and
\(4-2n\leq1-n\) for \(n\geq3\) to absorb
\(|Dw_0|^2=O_{2,\alpha,M}(|p|^2r^{4-2n})\) into the remainder.
Using \(A^{-1}=I-\eta_0^2p\otimes p\) and \eqref{eq:graph-w0-bounds}, multiplication
gives
\begin{equation}\label{eq:graph-model-flux}
 J_g(v_0)=\eta_0\left(p+A^{-1}Dw_0+\mathfrak c_pf_0p\right)
 +O_{2,\alpha,M}(|p|r^{1-n}).
\end{equation}
The vector density \(\eta_0(p+A^{-1}Dw_0+\mathfrak c_pf_0p)\) is divergence free by
\eqref{eq:graph-w0-equation}.
Taking its divergence proves \eqref{eq:graph-model-error}.
Finally, \eqref{eq:graph-w0-bounds}, the decay of \(g-\delta\), and a larger starting
radius prove \eqref{eq:graph-model-smallness}.
\end{proof}

On each auxiliary end define \(v_0:=0\).
There \eqref{eq:graph-model-error} holds with zero error, and
\eqref{eq:graph-model-smallness} is immediate.
The function \(v_0\) is defined only on the disjoint exterior ends.

\subsection{Global barriers}

\begin{lemma}\label{lem:graph-barriers}
There is \(\lambda_0(n)>0\) such that, for \(0<|p|\leq\lambda_0(n)\), the following assertions hold.
Choose a sufficiently large common radius \(R\) on all ends and put
\begin{equation}\label{eq:graph-radial}
 s:=n-2-\frac14,\qquad
 h(r):=b_0|p|R^{s+1}r^{-s}.
\end{equation}
There are globally defined, locally Lipschitz functions \(U^-\leq U^+\) on \(M\)
that are respectively a subsolution and a supersolution, and satisfy
\begin{equation}\label{eq:graph-tail-barriers}
 U^-=v_0-h,\qquad U^+=v_0+h\qquad\text{for }r\geq2R
\end{equation}
on every end.
They are constants on the compact core.
The radius \(R\) may depend on \(M\), whereas \(b_0\) and \(\lambda_0\) depend only on
\(n\).
\end{lemma}

\begin{proof}
We first verify the differential inequalities on the ends.
For the Euclidean radial function in \eqref{eq:graph-radial},
\begin{equation}\label{eq:graph-radial-laplacian}
 \Delta_\delta h=-\frac s4\frac h{r^2},\qquad
 |Dh|\leq C_n\frac h r,\qquad
 |D^2h|\leq C_n\frac h{r^2}.
\end{equation}
Set
\[
 \epsilon(R):=\max_E\sup_{x\in E,\,r(x)\geq R}
 \left(|g-\delta|_\delta+r|D(g-\delta)|_\delta\right),
\]
where the maximum is taken over all ends and \(D\) denotes differentiation
in the end coordinates. Asymptotic flatness gives \(\epsilon(R)\to0\) as
\(R\to\infty\). For sufficiently large \(R\), it follows that
\[
 \left|\Delta_gh+\frac s4\frac h{r^2}\right|
 \leq C_n\epsilon(R)\frac h{r^2}.
\]
For a covector \(q\), set
\[
 S(q):=\frac{q^\sharp\otimes q^\sharp}{1+|q|_g^2}.
\]
For \(f=v_0\) or \(h\), write \(f_i:=\nabla_i f\) and
\(f^i:=g^{ij}f_j\), with covariant derivatives taken with respect to \(g\).
Repeated upper and lower indices are summed. For \(\sigma\in\{-1,1\}\),
direct substitution into the nondivergence operator gives
\begin{align}\label{eq:graph-operator-difference}
 \mathcal P_g(v_0+\sigma h)
 &=\mathcal P_g(v_0)+\sigma\Delta_gh\notag\\
 &\quad-\left[
 \frac{(v_0^i+\sigma h^i)(v_0^j+\sigma h^j)}
 {1+g^{ab}(v_{0,a}+\sigma h_a)(v_{0,b}+\sigma h_b)}
 -\frac{v_0^iv_0^j}{1+g^{ab}v_{0,a}v_{0,b}}
 \right]\nabla_i\nabla_jv_0\notag\\
 &\quad-\sigma
 \frac{(v_0^i+\sigma h^i)(v_0^j+\sigma h^j)}
 {1+g^{ab}(v_{0,a}+\sigma h_a)(v_{0,b}+\sigma h_b)}
 \nabla_i\nabla_jh\notag\\
 &=\mathcal P_g(v_0)+\sigma\Delta_gh\notag\\
 &\quad-\left\langle S(dv_0+\sigma dh)-S(dv_0),
 \nabla_g^2v_0\right\rangle_g\notag\\
 &\quad-\sigma\left\langle S(dv_0+\sigma dh),
 \nabla_g^2h\right\rangle_g.
\end{align}
Here \(\langle\cdot,\cdot\rangle_g\) denotes complete contraction of the two tensor indices.
For covectors \(q_0,q_1\) of norm at most one at the same point,
\[
 |S(q_1)-S(q_0)|_g
 \leq C(|q_0|_g+|q_1|_g)|q_1-q_0|_g,
 \qquad |S(q_1)|_g\leq |q_1|_g^2.
\]
The first inequality follows by differentiating \(S\) along the segment
joining \(q_0\) to \(q_1\).
By \eqref{eq:graph-model-smallness}, \eqref{eq:graph-radial-laplacian},
and asymptotic flatness, for sufficiently large \(R\) we have
\[
 |dv_0|_g\leq C_n|p|,\qquad
 |\nabla_g^2v_0|_g\leq C_n\frac{|p|}{r},\qquad
 |dh|_g\leq C_n\frac hr,\qquad
 |\nabla_g^2h|_g\leq C_n\frac h{r^2}.
\]
Since \(h/r\leq b_0|p|\) for \(r\geq R\), taking
\(q_0=dv_0\) and \(q_1=dv_0+\sigma dh\) gives
\begin{align*}
 |S(dv_0+\sigma dh)-S(dv_0)|_g
 &\leq C_{n,b_0}|p|\frac hr,\\
 |S(dv_0+\sigma dh)|_g&\leq C_{n,b_0}|p|^2,
\end{align*}
provided \(|p|\) is sufficiently small depending only on \(n,b_0\).
Thus the last two terms in \eqref{eq:graph-operator-difference} satisfy
\begin{align*}
 \left|\left\langle S(dv_0+\sigma dh)-S(dv_0),
 \nabla_g^2v_0\right\rangle_g\right|
 &\leq C_{n,b_0}|p|\frac hr\frac{|p|}{r}
 =C_{n,b_0}|p|^2\frac h{r^2},\\
 \left|\left\langle S(dv_0+\sigma dh),
 \nabla_g^2h\right\rangle_g\right|
 &\leq C_{n,b_0}|p|^2\frac h{r^2}.
\end{align*}
Combining these bounds with \eqref{eq:graph-operator-difference} yields
\[
 \left|\mathcal P_g(v_0+\sigma h)-\mathcal P_g(v_0)-\sigma\Delta_gh\right|
 \leq C_{n,b_0}|p|^2\frac h{r^2}.
\]
By \eqref{eq:graph-model-error} and \eqref{eq:graph-model-smallness},
\(\mathcal P_g(v_0)=\sqrt{1+|dv_0|_g^2}\,\mathcal M_g(v_0)\)
satisfies \(|\mathcal P_g(v_0)|\leq C_M|p|r^{-n}\).
Dividing by \(h/r^2\) gives
\begin{equation}\label{eq:graph-error-ratio}
 \begin{aligned}
 \frac{|\mathcal P_g(v_0)|}{h/r^2}
 &\leq\frac{C_M|p|r^{-n}}{b_0|p|R^{s+1}r^{-s-2}}
 =\frac{C_M}{b_0}R^{-s-1}r^{s+2-n}\\
 &=\frac{C_M}{b_0}R^{5/4-n}r^{-1/4}
 =\frac{C_M}{b_0}R^{1-n}\left(\frac Rr\right)^{1/4}
 \leq\frac{C_M}{b_0}R^{1-n},
 \end{aligned}
\end{equation}
where we used \(s=n-2-\frac14\) and \(r\geq R\).
On auxiliary ends the residual is zero.
Consequently
\begin{equation}\label{eq:graph-barrier-error}
 \left|\mathcal P_g(v_0+\sigma h)+\sigma\frac s4\frac h{r^2}\right|
 \leq\left(C_n\epsilon(R)+C_{n,b_0}|p|^2
       +\frac{C_M}{b_0}R^{1-n}\right)\frac h{r^2}.
\end{equation}

We now choose \(b_0\), \(\lambda_0\), and \(R\), in this order, to ensure
the differential inequalities above and allow the barriers to be extended
across the compact core.
Choose \(b_0\) so large that
\begin{equation}\label{eq:graph-B-choice}
 \frac{b_0(1-2^{-s})}{2}>2D_n,
 \qquad b_02^{-s}>2D_n.
\end{equation}
Next choose \(\lambda_0(n)\) small enough that, for
\(0<|p|\leq\lambda_0\) and \(\sigma\in\{-1,1\}\),
\[
 |dv_0|_g\leq1,\qquad |d(v_0+\sigma h)|_g\leq1
 \quad\text{on }\{r\geq R\},
\]
and \(C_{n,b_0}\lambda_0^2<s/16\).
The gradient bounds follow from \eqref{eq:graph-model-smallness},
\eqref{eq:graph-radial-laplacian}, and \(h/r\leq b_0|p|\).
Finally choose \(R\) large enough that the remaining two coefficients in
\eqref{eq:graph-barrier-error} have sum less than \(s/16\).
All choices of \(R\) can be uniform for \(0<|p|\leq\lambda_0\).
This proves
\begin{equation}\label{eq:graph-strict-barriers}
 \mathcal M_g(v_0+h)<0,\qquad \mathcal M_g(v_0-h)>0\qquad(r\geq R).
\end{equation}

We have now constructed the barriers on each end. It remains to glue them
to constants across the compact core to obtain globally defined barriers.
For this purpose, set
\[
 H:=\frac{1+2^{-s}}2b_0|p|R=\frac{h(R)+h(2R)}2.
\]
Then \(h(2R)<H<h(R)\), and \eqref{eq:graph-B-choice} gives the margins
\[
 h(R)-H=H-h(2R)
 =\frac{b_0(1-2^{-s})}{2}|p|R>2D_n|p|R.
\]
Together with \(|v_0|\leq D_n|p|r\) from
\eqref{eq:graph-model-smallness}, these inequalities imply
\begin{align*}
 v_0+h&>H, &v_0-h&<-H &&\text{on }S_R,\\
 v_0+h&<H, &v_0-h&>-H &&\text{on }S_{2R}.
\end{align*}
Define, on each end,
\begin{equation}\label{eq:graph-gluing}
 \begin{aligned}
 U^+&:=
 \begin{cases}
 H,&\text{on the compact core},\\
 \min\{H,v_0+h\},&R<r<2R,\\
 v_0+h,&r\geq2R,
 \end{cases}\\
 U^-&:=
 \begin{cases}
 -H,&\text{on the compact core},\\
 \max\{-H,v_0-h\},&R<r<2R,\\
 v_0-h,&r\geq2R.
 \end{cases}
 \end{aligned}
\end{equation}
The strict inequalities on the two coordinate spheres imply that the
respective definitions agree in neighborhoods of these spheres.
On \(R\leq r\leq2R\), \eqref{eq:graph-B-choice} implies \(h>|v_0|\), so \(U^-\leq U^+\)
there. This ordering also holds on the core and the tails.

By the standard pasting property, \(U^-\) and \(U^+\) are respectively
a global weak subsolution and supersolution. The weak comparison principle
therefore applies to these barriers on bounded domains.
\end{proof}

\subsection{The exhaustion and its asymptotics}

\begin{proof}[Proof of Theorem~\ref{thm:affine-graph}]
Fix \(p\) and the barriers of Lemma~\ref{lem:graph-barriers}.
For \(T>2R\), remove the parts \(r>T\) from all ends and denote the resulting
connected compact domain by \(\Omega_T\).
For large \(T\), its boundary is smooth and strictly mean convex.
Lemma~\ref{lem:graph-dirichlet} gives \(u_T\) with
\begin{equation}\label{eq:graph-exhaustion}
 \begin{cases}
 \mathcal M_g(u_T)=0&\text{in }\Omega_T,\\
 u_T=v_0&\text{on the distinguished sphere},\\
 u_T=0&\text{on the other spheres}.
 \end{cases}
\end{equation}
The prescribed data lie between \(U^-\) and \(U^+\).
Comparison therefore gives
\begin{equation}\label{eq:graph-uniform-trapping}
 U^-\leq u_T\leq U^+\quad\text{on }\Omega_T.
\end{equation}
In particular, for every compact \(K\Subset M\), \(\sup_K|u_T|\) is bounded
independently of sufficiently large \(T\).
Spruck's interior gradient estimate \cite[Theorem~1.1]{Spruck2007}, followed by
interior elliptic regularity, gives bounds for all derivatives on \(K\),
independently of sufficiently large \(T\).
A diagonal subsequence converges smoothly on compact subsets to a global
solution \(u\).
Passing \eqref{eq:graph-uniform-trapping} to the limit proves
\begin{equation}\label{eq:graph-height-decay}
 |u-v_0|\leq h\quad\text{on }E_0\cap\{r\geq2R\},\qquad
 |u|\leq h\quad\text{on the other tails}.
\end{equation}

We next establish the corresponding estimates for derivatives.
Let \(x_L\) be a point with \(|x_L|=L\) on an end.
On the coordinate ball \(B_{L/4}(x_L)\), introduce the rescaled coordinates
\(x=x_L+Ly\), where \(|y|<1/4\), and set
\[
 g_L(y):=g(x_L+Ly),\qquad
 U_L(y):=\frac{u(x_L+Ly)}{L}.
\]
The function \(U_L\) solves \(\mathcal M_{g_L}(U_L)=0\).
The matrices \(g_L\) converge to \(\delta\) in \(C^{2,\alpha}\) as
\(L\to+\infty\), and
\eqref{eq:graph-height-decay} and \eqref{eq:graph-model-smallness} give
\(|U_L|\leq C\) on \(B_{1/4}(0)\), with \(C\) independent of \(L\).
Set \(\widetilde U_L:=U_L+C\). Since adding a constant does not change
the minimal graph equation,
\[
 \mathcal M_{g_L}(\widetilde U_L)=0,\qquad
 0\leq\widetilde U_L\leq2C\quad\text{on }B_{1/4}(0).
\]
Applying Spruck's interior gradient estimate to \(\widetilde U_L\)
gives a uniform bound for \(D\widetilde U_L\) on a smaller ball.
The standard interior H\"older estimate for the gradient, followed by
Schauder estimates for the uniformly elliptic equation, gives uniform
\(C^{3,\alpha}\) bounds for \(\widetilde U_L\) there, with constants independent of \(L\) for
sufficiently large \(L\).
Since \(D^j\widetilde U_L=D^jU_L\) for \(j\geq1\), these derivative bounds
also hold for \(U_L\).
For \(v_0=p\cdot x+w_0\), we have \(Dv_0=p+Dw_0\) and
\(D^jv_0=D^jw_0\) for \(j\geq2\). Thus \eqref{eq:graph-w0-bounds}
gives uniform \(C^{3,\alpha}\) bounds for \(v_0(x_L+Ly)/L\) on the
same smaller balls. On auxiliary ends these bounds are immediate since
\(v_0=0\).

Write \(z:=u-v_0\) on the distinguished end, and \(z:=u\) on an auxiliary end.
For each fixed \(x\), apply the fundamental theorem of calculus to
\(t\mapsto J_g^i(v_0+tz)(x)\), \(0\leq t\leq1\), where \(J_g\) is defined
in \eqref{eq:graph-flux-density}. The chain rule gives
\[
 J_g^i(u)-J_g^i(v_0)
 =\int_0^1\frac{d}{dt}J_g^i(v_0+tz)\,dt
 =a^{ij}\partial_jz,
\]
where \(a^{ij}\) is defined by
\begin{equation}\label{eq:graph-a-coefficients}
 a^{ij}(x):=\int_0^1
 \frac{\partial}{\partial\xi_j}
 \left(\sqrt{\det g}\frac{g^{i\ell}\xi_\ell}
 {\sqrt{1+g^{k\ell}\xi_k\xi_\ell}}\right)_
 {\xi=Dv_0+t(Du-Dv_0)}dt.
\end{equation}
Taking the coordinate divergence and using \(\mathcal M_g(u)=0\), we obtain
\begin{equation}\label{eq:graph-polarized}
 \partial_i(a^{ij}\partial_jz)=f,
 \qquad f:=-\sqrt{\det g}\,\mathcal M_g(v_0).
\end{equation}
The bounds for \(Du\) and \(Dv_0\) uniformly bound every covector
\(Dv_0+t(Du-Dv_0)\), \(0\leq t\leq1\), in
\eqref{eq:graph-a-coefficients}. Since \(g\) is uniformly comparable to
\(\delta\) on the ends, the matrices in that integral are uniformly
positive definite. Hence the operator
\(\partial_i(a^{ij}\partial_j\,\cdot)\) in \eqref{eq:graph-polarized}
is uniformly elliptic outside a compact set.
Moreover, \eqref{eq:graph-a-coefficients} depends smoothly on \(g\),
\(Du\), and \(Dv_0\). The \(C^{2,\alpha}\) bounds for \(g_L\) and the
\(C^{3,\alpha}\) bounds for \(U_L\) and \(v_0(x_L+Ly)/L\) therefore
give uniform \(C^{2,\alpha}\) bounds for \(a^{ij}(x_L+Ly)\) on smaller
fixed balls, with constants independent of \(L\).
Equation~\eqref{eq:graph-model-error} gives \(f=O_{1,\alpha}(r^{-n})\) on \(E_0\),
while \(f=0\) on the other ends.

For \(Z_L(y):=z(x_L+Ly)\), equation~\eqref{eq:graph-polarized} becomes
\[
 \partial_{y_i}\left(a^{ij}(x_L+Ly)\partial_{y_j}Z_L\right)
 =L^2f(x_L+Ly).
\]
By \eqref{eq:graph-height-decay}, \(|z|\leq h\) on the tails. Since
\(3L/4\leq|x_L+Ly|\leq5L/4\) for \(|y|<1/4\), this gives
\(\|Z_L\|_{C^0(B_{1/4}(0))}\leq CL^{-s}\) for sufficiently large \(L\),
with \(C\) independent of \(L\).
The linear interior Schauder estimate on nested balls therefore gives
\begin{equation}\label{eq:graph-scaled-schauder}
 \|Z_L\|_{C^{3,\alpha}}
 \leq C\left(\|Z_L\|_{C^0}+L^{2-n}\right)
 \leq C\left(L^{-s}+L^{2-n}\right)
 \leq CL^{-s}.
\end{equation}
The norms on the right are taken on a slightly larger fixed ball.
Covering each annulus by finitely many such balls yields
\begin{equation}\label{eq:graph-derivative-decay}
 u-v_0=O_{3,\alpha}(r^{-s})\quad\text{on }E_0,
 \qquad u=O_{3,\alpha}(r^{-s})\quad\text{on other ends}.
\end{equation}
This proves \eqref{eq:graph-main-asymptotics}.

We next derive the leading harmonic term in the expansion on each
auxiliary end.
On such an end the equation is
\[
 \partial_i(B^{ij}\partial_ju)=0,\qquad
 B^{ij}:=\frac{\sqrt{\det g}\,g^{ij}}{\sqrt{1+|du|_g^2}}.
\]
By \eqref{eq:af-definition} and \eqref{eq:graph-derivative-decay},
\[
 B-I=O_{2,\alpha}(r^{-\kappa}),\qquad
 \kappa:=\min\{\tau_E,2n-\tfrac52\}>\frac12.
\]
Consequently
\begin{equation}\label{eq:graph-aux-poisson}
 \Delta_\delta u
 =\partial_i\bigl((\delta^{ij}-B^{ij})\partial_ju\bigr)
 =O_{1,\alpha}(r^{-n-\kappa+1/4})
 =O_{1,\alpha}(r^{-n-1/4}).
\end{equation}
Identify the end with \(\{x\in\mathbb R^n:|x|>R_0\}\) in its asymptotic
coordinates. Choose \(R_0<R_1<R_2\) and a smooth cutoff
\(\chi:\mathbb R^n\to[0,1]\) such that
\[
 \chi=0\quad\text{for }|x|\leq R_1,\qquad
 \chi=1\quad\text{for }|x|\geq R_2.
\]
Define
\[
 \widetilde u(x):=
 \begin{cases}
 \chi(x)u(x),&|x|>R_0,\\
 0,&|x|\leq R_0.
 \end{cases}
\]
Since \(\chi\) vanishes near the inner boundary, \(\widetilde u\) is smooth
on \(\mathbb R^n\). Moreover, \(\widetilde u=u\) for \(|x|\geq R_2\),
so this extension preserves the asymptotic behavior of \(u\).
Set \(F:=\Delta_\delta\widetilde u\). On \(|x|>R_0\), the product rule gives
\[
 F=\chi\Delta_\delta u+2D\chi\cdot Du+u\Delta_\delta\chi.
\]
The last two terms are supported in \(\{R_1\leq|x|\leq R_2\}\).
Thus \(F\) is smooth on \(\mathbb R^n\) and, by
\eqref{eq:graph-aux-poisson}, satisfies
\(F=O_{1,\alpha}(r^{-n-1/4})\) as \(r\to\infty\).
Let
\[
 \Gamma(x):=-\frac{|x|^{2-n}}{(n-2)\omega_{n-1}},\qquad
 \Delta_\delta\Gamma=\delta_0,
\]
where \(\omega_{n-1}\) is the area of the unit sphere.
Since \(\widetilde u(x)\to0\) as \(|x|\to\infty\) and
\(F=\Delta_\delta\widetilde u=O(r^{-n-1/4})\), the Newton potential
representation gives
\[
 \widetilde u(x)=\int_{\mathbb R^n}\Gamma(x-y)F(y)\,dy.
\]
We use this formula to obtain a sharper asymptotic expansion of \(u\)
on the end.
The decay of \(F\) implies
\[
 \int_{\mathbb R^n}|y|^{1/8}|F(y)|\,dy<\infty.
\]
Set \(r:=|x|\). To estimate
\[
 \widetilde u(x)-\Gamma(x)\int_{\mathbb R^n}F(y)\,dy
 =\int_{\mathbb R^n}[\Gamma(x-y)-\Gamma(x)]F(y)\,dy,
\]
we split the integral into \(|y|<r/2\) and \(|y|\geq r/2\).
On the first region, the mean value theorem gives
\[
 |\Gamma(x-y)-\Gamma(x)|
 \leq C_nr^{1-n}|y|
 \leq C_nr^{2-n-1/8}|y|^{1/8}.
\]
Hence
\begin{align*}
 \left|\int_{|y|<r/2}[\Gamma(x-y)-\Gamma(x)]F(y)\,dy\right|
 &\leq Cr^{2-n-1/8}\int_{\mathbb R^n}|y|^{1/8}|F(y)|\,dy\\
 &\leq Cr^{2-n-1/8}.
\end{align*}
On the second region, \(|F(y)|\leq C|y|^{-n-1/4}\) implies
\[
 \int_{|y|\geq r/2}|F(y)|\,dy
 \leq C\int_{r/2}^{\infty}\rho^{-1-1/4}\,d\rho
 \leq Cr^{-1/4}.
\]
To handle the singularity of \(\Gamma(x-y)\), split this region further
according to whether \(|x-y|<r/2\) or \(|x-y|\geq r/2\). Then
\begin{align*}
 \int_{\substack{|y|\geq r/2\\
 |x-y|<r/2}}
 |\Gamma(x-y)||F(y)|\,dy
 &\leq Cr^{-n-1/4}\int_{|w|<r/2}|w|^{2-n}\,dw
 \leq Cr^{2-n-1/4},\\
 \int_{\substack{|y|\geq r/2\\
 |x-y|\geq r/2}}
 |\Gamma(x-y)||F(y)|\,dy
 &\leq Cr^{2-n}\int_{|y|\geq r/2}|F(y)|\,dy
 \leq Cr^{2-n-1/4}.
\end{align*}
The subtracted term satisfies the same bound,
\[
 |\Gamma(x)|\int_{|y|\geq r/2}|F(y)|\,dy
 \leq Cr^{2-n-1/4}.
\]
Combining these estimates and using \(\widetilde u=u\) for
\(|x|\geq R_2\) yields
\begin{equation}\label{eq:graph-newton-expansion}
 u(x)=\Gamma(x)\int_{\mathbb R^n}F(y)\,dy+O(r^{2-n-1/8}).
\end{equation}
Set
\[
 b_E:=-\frac{1}{(n-2)\omega_{n-1}}\int_{\mathbb R^n}F(y)\,dy.
\]
Since \(r^{2-n}\) is harmonic away from the origin, on \(r\geq R_2\) we have
\[
 \Delta_\delta(u-b_Er^{2-n})=F=O_{1,\alpha}(r^{-n-1/4}).
\]
On rescaled annuli the source term is therefore
\(O_{1,\alpha}(L^{2-n-1/4})\), while \eqref{eq:graph-newton-expansion}
gives \(u-b_Er^{2-n}=O(L^{2-n-1/8})\).
Interior Schauder estimates on nested annuli yield
\(u-b_Er^{2-n}=O_{3,\alpha}(r^{2-n-1/8})\).
This proves \eqref{eq:graph-aux-asymptotics}.

The gradient of \(u\) is bounded on all ends by the expansions and on the
compact core by smoothness.
The graph metric is complete because \(g+du\otimes du\geq g\).
If two solutions have the prescribed end limits after subtracting the same
model, their difference tends to zero at every end and solves a linear
elliptic equation obtained by polarizing the minimal operator.
Applying the maximum principle on \(\Omega_T\) and letting \(T\to\infty\)
proves uniqueness.
Uniqueness implies local smooth convergence of the entire family
\((u_T)\) to this solution.
\end{proof}

\section{Scalar curvature of a minimal graph}\label{sec:graph-geometry}

The key identity in this section is \eqref{eq:3.5}. It shows that
the scalar curvature of the conformal graph metric contains a
nonnegative quadratic term in the Hessian of the graph function.
This allows us to derive Hessian estimates for the coordinate
functions from the mass--energy inequality.

Let \((M^n,g)\) satisfy the hypotheses of
Theorem~\ref{thm:affine-graph}, and let \(u:=u_p\) be one of its
minimal graph functions.  Define the lapse and graph metric by
\begin{equation}
 \eta:=(1+|du|_g^2)^{-1/2},\qquad
 \bar g:=g+du\otimes du.
\label{eq:3.1}
\end{equation}
Equip the graph with the downward unit normal
\(\nu:=\eta(\nabla u,-1)=-N_u\), where \(N_u:=\eta(-\nabla u,1)\) is the upward unit normal,
and define its second fundamental form by
\[
 h(U,V):=\langle\nabla^{g+dt^2}_U\nu,V\rangle.
\]
Then \(h=\eta\nabla^2u\), and
\[
 \eta(\nabla^ju)h_{ij}=-(d\log\eta)_i.
\]
Vertical translations preserve minimality, so
\(\eta=-\langle\nu,\partial_t\rangle\) satisfies the Jacobi equation
\[
 \bar\Delta\eta+
 \bigl(|h|_{\bar g}^2+\operatorname{Ric}_{g+dt^2}(\nu,\nu)\bigr)\eta=0.
\]
The Gauss equation reads
\(R_{\bar g}=R_g-2\operatorname{Ric}_{g+dt^2}(\nu,\nu)-|h|_{\bar g}^2\).
Eliminating the ambient Ricci term gives the minimal graph case of the
Schoen--Yau identity \cite{SchoenYau1981}:

\begin{equation}
 R_{\bar g}
 =R_g+|h|_{\bar g}^2+2|d\log\eta|_{\bar g}^2
 +2\bar\Delta\log\eta.
\label{eq:3.3}
\end{equation}
Here \(\bar\Delta\) is the Laplacian of \(\bar g\).

To cancel the Laplacian term, set

\begin{equation}
 \widetilde g:=\eta^{2/(n-1)}\bar g.
\label{eq:3.4}
\end{equation}

The conformal scalar curvature formula gives
\[
 R_{\widetilde g}
 =\eta^{-2/(n-1)}\left[
 R_{\bar g}-2\bar\Delta\log\eta
 -\frac{n-2}{n-1}|d\log\eta|_{\bar g}^2\right].
\]
Substituting \eqref{eq:3.3} yields

\begin{equation}
 R_{\widetilde g}
 =\eta^{-2/(n-1)}
 \left(
 R_g+|h|_{\bar g}^2+
 \frac n{n-1}|d\log\eta|_{\bar g}^2
 \right).
\label{eq:3.5}
\end{equation}

The corresponding volume forms are
\begin{equation}
 dV_{\bar g}=\eta^{-1}dV_g,\qquad
 dV_{\widetilde g}=\eta^{1/(n-1)}dV_g.
\label{eq:3.6}
\end{equation}

In particular, \(R_g\geq0\) implies \(R_{\widetilde g}\geq0\).
The gradient \(du\) is bounded on every end by
Theorem~\ref{thm:affine-graph} and on the compact core by smoothness.
Thus \(\inf_M\eta>0\), and
\[
 \widetilde g\geq(\inf_M\eta)^{2/(n-1)}g
\]
proves completeness.  This lower bound may depend on \((M,g,u)\), whereas
the comparisons on regions of small coordinate defect have constants
depending only on the dimension.

Throughout the following argument, up to Subsection~\ref{subsec:density},
assume that \(R_g\geq0\), \(R_g\in L^1(M)\), and the distinguished end has
harmonic asymptotics. On the distinguished end,
\eqref{eq:graph-main-asymptotics} and \eqref{eq:graph-w0-bounds} give
\(Du=p+O_{2,\alpha}(r^{2-n})\). Together with
\(g=I+O_{2,\alpha}(r^{2-n})\), this implies
\[
 \bar g=A+O_{2,\alpha}(r^{2-n}),\qquad
 \eta=\eta_0+O_{2,\alpha}(r^{2-n}),
\]
where \(A=I+p\otimes p\) and \(\eta_0=(1+|p|^2)^{-1/2}\) are defined in
\eqref{eq:graph-parameters}. Hence, in the original end coordinates,
\[
 \widetilde g=\eta_0^{2/(n-1)}A+O_{2,\alpha}(r^{2-n}).
\]
In the linear coordinates \(y=\eta_0^{1/(n-1)}A^{1/2}x\), this becomes
\[
 \widetilde g_{ab}(y)=\delta_{ab}+O_{2,\alpha}(|y|^{2-n}).
\]
The decay order is unchanged because the coordinate transformation is
linear and invertible, with \(|y|\) comparable to \(|x|\).
Thus the distinguished end is AF for \(\widetilde g\).
On each auxiliary end,
\eqref{eq:graph-aux-asymptotics} gives
\(du=O_{2,\alpha}(r^{1-n})\) and
\(\widetilde g-g=O_{2,\alpha}(r^{2-2n})\), so that end remains AF.

Moreover, \eqref{eq:3.5}--\eqref{eq:3.6} give
\begin{equation}
 R_{\widetilde g}\,dV_{\widetilde g}
 =\eta^{-1/(n-1)}
 \left(
 R_g+|h|_{\bar g}^{2}
 +\frac n{n-1}|d\log\eta|_{\bar g}^{2}
 \right)dV_g.
 \label{eq:tilted-density}
\end{equation}
The graph expansions give \(h,d\log\eta=O(r^{1-n})\) on the
distinguished end and faster decay on the auxiliary ends.  Both terms
therefore lie in \(L^2(M,dV_g)\).  Since \(\eta\) is bounded away from
zero and \(R_g\in L^1(M,dV_g)\), the preceding identity proves
\(R_{\widetilde g}\in L^1(M,dV_{\widetilde g})\).
Thus each tilted metric satisfies the geometric hypotheses of
Theorem~\ref{thm:positive-mass} and of the scalar construction below.

\section{ADM mass of the tilted minimal graphs}\label{sec:tilted-mass}

In this section, we prove the mass formula
\eqref{eq:tilted-mass-formula}. After normalizing the limiting metric by
a linear change of coordinates, we compute the ADM contributions of the
leading terms in the metric expansion. The contribution involving \(w_0\)
is determined by a spherical integral that can be evaluated using its
equation. Combining these contributions cancels the direction-dependent
spherical integrals and gives \eqref{eq:tilted-mass-formula}. This formula
will allow us to apply the mass--energy inequality to \(\widetilde g\)
and obtain integral bounds for the Hessians of the graph functions.

\begin{proposition}[Mass of a tilted graph]\label{prop:tilted-mass}
Let \((M^n,g)\), \(n\geq3\), be a smooth, connected, complete, oriented
manifold without boundary and with finitely many AF ends.  Suppose the
distinguished end has harmonic asymptotics.  For
\(0<|p|\leq\lambda_0(n)\), let \(u:=u_p\) be the minimal graph function
from Theorem~\ref{thm:affine-graph}.  Define
\[
 \widetilde g
 :=(1+|du|_g^2)^{-1/(n-1)}(g+du\otimes du).
\]
Then \((M^n,\widetilde g)\) is asymptotically flat in suitable
coordinates, and its ADM mass is
\begin{equation}
 m(\widetilde g)
 =(1+|p|^2)^{1/[2(n-1)]}m(g).
\label{eq:tilted-mass-formula}
\end{equation}
\end{proposition}

\begin{proof}
\noindent\textbf{Step 1. Expand the metric and normalize the coordinates.}
We identify the leading terms that can contribute to the ADM mass and
choose coordinates in which the limiting metric is Euclidean.
Write
\begin{equation}
 \begin{gathered}
 m:=m(g),\qquad \eta:=(1+|du|_g^2)^{-1/2},\qquad
 \eta_0:=(1+|p|^2)^{-1/2},\\
 A:=\delta+p\otimes p,\qquad
 \kappa:=\eta_0^2|p|^2,\qquad
 \mathfrak c_p:=\frac{n-2+\kappa}{2}.
 \end{gathered}
 \label{eq:4.2}
\end{equation}
Harmonic asymptotics give
\begin{equation}
 g=(1+f_0)\delta+o_{2,\alpha}(r^{2-n}),\qquad
 f_0:=\frac{2m}{n-2}r^{2-n}.
 \label{eq:4.1}
\end{equation}

Put \(\widehat u:=u-p\cdot x\). Theorem~\ref{thm:affine-graph} gives
\(\widehat u=w_0+O_{3,\alpha}(r^{2-n+1/4})\), and hence
\begin{equation}
 d\widehat u=dw_0+o_{1,\alpha}(r^{2-n}),\qquad
 D^2\widehat u=D^2w_0+o_{0,\alpha}(r^{1-n}),
 \label{eq:leading-w-expansion}
\end{equation}
where \(w_0\) is homogeneous of degree \(3-n\).
Using \eqref{eq:4.1} and \eqref{eq:leading-w-expansion}, Taylor expansion gives
\begin{align*}
 \sqrt{\det g}\,g^{-1}
 &=\left(1+\frac{n-2}{2}f_0\right)I+o_{2,\alpha}(r^{2-n}),\\
 (1+|p+D\widehat u|_g^2)^{-1/2}
 &=\eta_0\left(1-\eta_0^2p\cdot Dw_0
 +\frac12\eta_0^2|p|^2f_0\right)+o_{1,\alpha}(r^{2-n}).
\end{align*}
Since \(f_0,Dw_0=O_{1,\alpha}(r^{2-n})\), their quadratic products are
\(o_{1,\alpha}(r^{2-n})\). Moreover,
\[
 A^{-1}Dw_0=Dw_0-\eta_0^2(p\cdot Dw_0)p,
 \qquad \mathfrak c_p=\frac{n-2+\eta_0^2|p|^2}{2}.
\]
Multiplying the expansions and collecting the terms of order \(r^{2-n}\)
therefore gives
\begin{equation}
 \sqrt{\det g}\,
 \frac{g^{-1}(p+D\widehat u)}
 {\sqrt{1+|p+D\widehat u|_g^2}}
 =\eta_0
 \left(p+A^{-1}Dw_0+\mathfrak c_p f_0p\right)
 +o_{1,\alpha}(r^{2-n}).
 \label{eq:minimal-flux-expansion}
\end{equation}
Since \(u\) solves the minimal graph equation, the coordinate divergence
of the left-hand side of \eqref{eq:minimal-flux-expansion} vanishes.
The divergence of the leading expression on the right is
homogeneous of degree \(1-n\), and the remainder is \(o(r^{1-n})\).
Hence
\begin{equation}
 \operatorname{div}_{\delta}
 \bigl(A^{-1}Dw_0+\mathfrak c_p f_0p\bigr)=0
 \qquad\text{on }\mathbb R^n\setminus\{0\}.
 \label{eq:4.3}
\end{equation}

Set
\[
 \widehat g:=
 \left(\frac{\eta}{\eta_0}\right)^{2/(n-1)}
 (g+du\otimes du).
\]
Then \(\widehat g\to A\) as \(|x|\to\infty\) on the distinguished end, and
\(\widetilde g=\eta_0^{2/(n-1)}\widehat g\).  We first compute the mass
of \(\widehat g\) relative to \(A\). Expanding
\(\eta=(1+|p+D\widehat u|_g^2)^{-1/2}\) gives
\begin{equation}
 \frac{\eta}{\eta_0}-1
 =\frac\kappa2f_0-\eta_0^2p\cdot Dw_0
 +o_{1,\alpha}(r^{2-n}),
 \label{eq:4.4}
\end{equation}
so
\begin{equation}
 \begin{aligned}
 \widehat g-A
 ={}& f_0\delta+p\otimes Dw_0+Dw_0\otimes p\\
 &+\frac2{n-1}
 \left(\frac\kappa2f_0-\eta_0^2p\cdot Dw_0\right)A
 +o_{1,\alpha}(r^{2-n}).
 \end{aligned}
\label{eq:4.5}
\end{equation}
Choose the first coordinate axis in the \(p\)-direction, so that
\(p=|p|e_1\). Then
\[
 A=\operatorname{diag}(1+|p|^2,1,\ldots,1),\qquad
 A^{-1}=\operatorname{diag}(1-\kappa,1,\ldots,1).
\]
Introduce
\(y:=A^{1/2}x\), \(R:=|y|\), and \(\theta:=y/R\).
This change of coordinates normalizes \(A\) to \(\delta\).
In the following mass calculations, \(S_R:=\{|y|=R\}\), and
\(\nu\) and \(dA\) denote its Euclidean outward unit normal and area measure.

\medskip
\noindent\textbf{Step 2. Remove the terms with zero mass contribution.}
We show that the remainder and the cross term do not affect the ADM mass.
It is easy to see that the remainder \(o_{1,\alpha}(r^{2-n})\)
makes no contribution to the ADM mass.

The sum \(p\otimes Dw_0+Dw_0\otimes p\) also makes no contribution
to the ADM mass. Indeed, in the original \(x\)-coordinates,

\begin{equation}
 p\otimes Dw_0+Dw_0\otimes p
 =\mathcal L_{V_0}A,\qquad
 V_0:=w_0A^{-1}p.
 \label{eq:pure-gauge-term}
\end{equation}
In the normalized \(y\)-coordinates, still writing \(V_0\) for the
transformed vector field, this term becomes \(\mathcal L_{V_0}\delta\).
Its ADM boundary integral vanishes by the divergence theorem:
\[
 \int_{S_R}(\Delta V_0
 -\nabla\operatorname{div}V_0)\cdot\nu\,dA=0.
\]
Here all operators are Euclidean in \(y\), and we extend \(V_0\)
smoothly into the ball without changing it near \(S_R\), using
\(\operatorname{div}(\Delta V_0-\nabla\operatorname{div}V_0)=0\).

\medskip
\noindent\textbf{Step 3. Compute the two remaining contributions.}
We compute the contribution of \(f_0\delta\), then use the equation for
\(w_0\) to evaluate the contribution of the scalar multiple of \(A\).
For these calculations, put
\[
 \rho^2:=\theta\cdot A^{-1}\theta=1-\kappa\theta_1^2,\qquad
 J_0:=\int_{S^{n-1}}\rho^{-n}\,d\theta,\qquad
 J_2:=\int_{S^{n-1}}\theta_1^2\rho^{-n}\,d\theta.
\]
The identity
\[
 \int_{S^{n-1}}(\theta\cdot B\theta)^{-n/2}\,d\theta
 =\omega_{n-1}(\det B)^{-1/2}
\]
holds for every positive definite matrix \(B\).  Taking
\(B=A^{-1}\) and using \(\det(A^{-1})=\eta_0^2\) gives
\begin{equation}
 J_0=\eta_0^{-1}\omega_{n-1}.
 \label{eq:4.6a}
\end{equation}
In the \(y\)-coordinates,
\begin{equation}
 f_0=\frac{2m}{n-2}R^{2-n}\rho^{2-n},\qquad
 \int_{S^{n-1}}\rho^{2-n}\,d\theta=J_0-\kappa J_2.
 \label{eq:ellipsoidal-identities}
\end{equation}
Here and below, the same symbols \(f_0\) and \(w_0\) denote the
functions expressed in the \(y\)-coordinates. In particular,
\(r=R\rho\).

We now compute the ADM mass contribution of \(f_0\delta\) in
\eqref{eq:4.5}. This term becomes \(f_0A^{-1}\) in the \(y\)-coordinates,
and direct differentiation in the ADM formula gives
\[
 \frac{m}{(n-1)\omega_{n-1}}
 \left\{\operatorname{tr}(A^{-1})
 \int_{S^{n-1}}\rho^{2-n}d\theta
 -\int_{S^{n-1}}\rho^{-n}
 (\theta\cdot A^{-2}\theta)d\theta\right\}.
\]
Using \(\operatorname{tr}(A^{-1})=n-\kappa\),
\(\theta\cdot A^{-2}\theta=1-(2\kappa-\kappa^2)\theta_1^2\), and
\eqref{eq:ellipsoidal-identities}, this contribution simplifies to
\begin{equation}
 \frac{m}{(n-1)\omega_{n-1}}
 \left[(n-1-\kappa)J_0-\kappa(n-2)J_2\right].
\label{eq:4.6b}
\end{equation}

Since \(y=A^{1/2}x\), one has
\(p\cdot D_x=(A^{1/2}p)\cdot D_y\).  Set
\[
 \bar p:=A^{1/2}p,\qquad
 z_0:=\bar p\cdot D_yw_0=p\cdot D_xw_0.
\]

It remains to compute the ADM mass contribution of
\(\frac1{n-1}(\kappa f_0-2\eta_0^2z_0)\delta\), the scalar term in
\eqref{eq:4.5} expressed in the normalized \(y\)-coordinates.
For this calculation, we first determine the spherical mean of \(z_0\).
Here \(D_y\) and \(\Delta_y\) denote the Euclidean gradient and
Laplacian in the normalized \(y\)-coordinates.  Equation \eqref{eq:4.3}
becomes \(\Delta_yw_0=-\mathfrak c_p\bar p\cdot D_yf_0\).  Set
\(\varphi(\theta):=w_0(1,\theta)\).  By homogeneity,
\(w_0(R,\theta)=R^{3-n}\varphi(\theta)\).
Homogeneity and integration by parts give, with \(z_0\) evaluated at \(R=1\),
\[
 \Delta_yw_0
 =R^{1-n}\bigl(
 \Delta_{S^{n-1}}\varphi-(n-3)\varphi\bigr),
 \qquad
 \int_{S^{n-1}}z_0\,d\theta
 =2\int_{S^{n-1}}(\bar p\cdot\theta)\varphi(\theta)\,d\theta.
\]
Multiply the spherical equation by \(\bar p\cdot\theta\) and integrate
over \(S^{n-1}\). Since
\(\Delta_{S^{n-1}}(\bar p\cdot\theta)
=-(n-1)(\bar p\cdot\theta)\), integration by parts on the sphere shows
that the integrated left-hand side is
\(-(n-2)\int_{S^{n-1}}z_0\,d\theta\).  On the other hand, evaluating at
\(R=1\),
\[
 D_yf_0=-2m\rho^{-n}A^{-1}\theta,
 \qquad
 \int_{S^{n-1}}(\bar p\cdot\theta)\,\bar p\cdot D_yf_0\,d\theta
 =-2m|p|^2J_2,
\]
because \(A^{-1}\bar p=(1-\kappa)\bar p\) and
\((1-\kappa)|\bar p|^2=|p|^2\). Multiplying the last integral by
\(-\mathfrak c_p\) and equating the two sides therefore yields
\begin{equation}
 \int_{S^{n-1}}z_0\,d\theta
 =-\frac{2\mathfrak c_p m|p|^2}{n-2}J_2.
 \label{eq:z0-average}
\end{equation}
Since \(z_0\) is homogeneous of degree \(2-n\),
\begin{equation}
 \int_{S_R}\partial_Rz_0\,dA
 =2\mathfrak c_p m|p|^2J_2.
 \label{eq:z0-flux}
\end{equation}
These formulas include \(n=3\).

We now compute the ADM mass contribution of the scalar multiple of \(A\)
in \eqref{eq:4.5}. In the normalized \(y\)-coordinates, this term is
\[
 \sigma\delta,\qquad
 \sigma:=\frac1{n-1}
 \left(\kappa f_0-2\eta_0^2z_0\right).
\]
For the scalar perturbation \(\sigma\delta\), the linear ADM
contribution is
\[
 -\frac1{2\omega_{n-1}}
 \lim_{R\to\infty}\int_{S_R}\partial_R\sigma\,dA.
\]
By \eqref{eq:ellipsoidal-identities},
\[
 \int_{S_R}\partial_Rf_0\,dA
 =-2m\int_{S^{n-1}}\rho^{2-n}\,d\theta
 =-2m(J_0-\kappa J_2).
\]
Using \eqref{eq:z0-flux}, \(\eta_0^2|p|^2=\kappa\), and
\(2\mathfrak c_p=n-2+\kappa\), we obtain the mass contribution
\begin{equation}
 \frac{m\kappa}{(n-1)\omega_{n-1}}
 \left[J_0+(n-2)J_2\right].
 \label{eq:4.6c}
\end{equation}

\medskip
\noindent\textbf{Step 4. Combine the contributions and restore the scaling.}
We first obtain the mass of \(\widehat g\), then account for the constant
factor relating \(\widetilde g\) to \(\widehat g\).
Adding \eqref{eq:4.6b} and \eqref{eq:4.6c} cancels the \(J_2\)-terms.
Equation~\eqref{eq:4.6a} then yields
\begin{equation}
 m(\widehat g)=\eta_0^{-1}m.
 \label{eq:4.6}
\end{equation}

Finally, \(\widetilde g=\eta_0^{2/(n-1)}\widehat g\).  If a metric and
its constant limiting background are both multiplied by \(c>0\), then
the ADM mass relative to the scaled background is
\(c^{(n-2)/2}\) times the original mass.
Hence
\[
 \begin{aligned}
 m(\widetilde g)
 &=\eta_0^{(n-2)/(n-1)}\eta_0^{-1}m\\
 &=\eta_0^{-1/(n-1)}m
 =(1+|p|^2)^{1/[2(n-1)]}m.
 \end{aligned}
\]
\end{proof}

For the remainder of the proof, fix the slope
\(\lambda:=\lambda_0(n)/2\).  For the standard orthonormal basis
\(\{e_a\}_{a=1}^n\) of the chosen AF coordinates, let
\(\widetilde g_a\) be the tilted metric associated with \(u_{\lambda e_a}\).
Then the tilted mass formula gives

\begin{equation}
 \sum_{a=1}^n m(\widetilde g_a)
 =n(1+\lambda^2)^{1/[2(n-1)]}m(g).
\label{eq:4.8}
\end{equation}

\section{The scalar conformal potential}\label{sec:scalar-potential}

In this section, we prove the mass--energy estimate
\eqref{eq:scalar-energy}. We first construct a solution \(w\) of the
conformal Laplacian equation with prescribed limits on all ends, and
then establish the existence of its normal derivative limits and the
conformal mass formula in Proposition~\ref{prop:conformal-mass}.
Applying this formula and the positive mass inequality to positive
affine combinations of \(w\) and \(1\) gives the energy bound in
Proposition~\ref{prop:scalar-energy}. This use of the positive mass
theorem to control the energy of the conformal factor is reminiscent
of Bray's argument in his proof of the Riemannian Penrose inequality
\cite{Bray2001}.

Let \((N^n,\gamma)\) be a smooth complete asymptotically flat manifold
with finitely many ends, as in Section~\ref{sec:preliminaries}, and
distinguish one end \(E_0\).  Assume
\[
 R_\gamma\geq0,\qquad R_\gamma\in L^1(N).
\]
Set
\begin{equation}
 a_n:=\frac{4(n-1)}{n-2},\qquad
 \kappa_n:=a_n^{-1}=\frac{n-2}{4(n-1)},\qquad
 L_\gamma:=-a_n\Delta_\gamma+R_\gamma,
 \label{eq:conformal-laplacian}
\end{equation}
where \(\Delta_\gamma:=\operatorname{div}_\gamma\nabla\).
Choose an AF decay rate \(\tau>(n-2)/2\) valid on every end, decreasing
the individual rates if necessary, and fix
\begin{equation}
 \frac{n-2}{2}<\sigma<\min\{\tau,n-2\}.
 \label{eq:scalar-weight}
\end{equation}

\begin{lemma}\label{lem:scalar-potential}
There is a unique smooth function \(w\) satisfying
\begin{equation}
 L_\gamma w=0,\qquad
 w\longrightarrow1\text{ on }E_0,\qquad
 w\longrightarrow0\text{ on every auxiliary end}.
 \label{eq:scalar-potential}
\end{equation}
It satisfies \(0<w\leq1\), and
\begin{equation}
 w-1=O_{2,\alpha}(r^{-\sigma})\text{ on }E_0,\qquad
 w=O_{2,\alpha}(r^{-\sigma})\text{ on every auxiliary end}.
 \label{eq:scalar-potential-decay}
\end{equation}
\end{lemma}

\begin{proof}
Let \(\chi\) be smooth, equal to one on a tail of \(E_0\), and equal
to zero on the tails of all auxiliary ends.  The AF estimates give
\(R_\gamma=O_{0,\alpha}(r^{-\tau-2})\), so
\(L_\gamma\chi\in C^{0,\alpha}_{-\sigma-2}(N)\).
The Fredholm theory for elliptic operators on AF manifolds gives an
operator of index zero
\[
 L_\gamma:C^{2,\alpha}_{-\sigma}(N)
 \longrightarrow C^{0,\alpha}_{-\sigma-2}(N)
\]
for \(0<\sigma<n-2\), by
\cite{CantorBrill1981,Bartnik1986,LockhartMcOwen1985}.
Its kernel is trivial.  Indeed, if \(L_\gamma v=0\) and
\(v\in C^{2,\alpha}_{-\sigma}(N)\), integration by parts on an
exhaustion gives
\[
 a_n\int_N|dv|_\gamma^2\,dV_\gamma
 +\int_N R_\gamma v^2\,dV_\gamma=0.
\]
The boundary terms vanish because, on each end,
\[
 \int_{S_r}|v|\,|dv|_\gamma\,dA_\gamma
 =O(r^{n-2-2\sigma})=o(1).
\]
Thus \(v\) is constant, and its decay gives \(v=0\).  The operator is
therefore invertible, and
\[
 w:=\chi-L_\gamma^{-1}(L_\gamma\chi)
\]
has the asserted equation and decay.

The maximum principle, applied on large compact exhaustions, gives
\(w\geq0\).  Since \(L_\gamma(1-w)=R_\gamma\geq0\) and \(1-w\)
has nonnegative limits at every end, it also gives \(w\leq1\).
The strong maximum principle and the limit on \(E_0\) imply \(w>0\).
Finally, the difference of two solutions with the prescribed end limits
tends to zero at every end, so the maximum principle proves uniqueness.
\end{proof}

\subsection{Flux and conformal change of mass}

For a function \(u\) on an end \(E\), define its outward flux,
when the limit exists, by
\begin{equation}
 \mathcal F_E(u):=
 \lim_{r\to\infty}\int_{S_r(E)}\partial_{\nu_\gamma}u\,dA_\gamma.
 \label{eq:scalar-flux}
\end{equation}
Here \(\nu_\gamma\) is the outward \(\gamma\)-unit normal to the
coordinate sphere. If \(u=c+O_{2,\alpha}(r^{-\sigma})\) on \(E\)
for a constant \(c\), then \(\sigma+\tau>n-2\) gives
\[
 \int_{S_r}\partial_{\nu_\gamma}u\,dA_\gamma
 -\int_{S_r}\partial_{\nu_\delta}u\,dA_\delta
 =O(r^{n-2-\sigma-\tau})=o(1).
\]
Thus replacing \(\nu_\gamma\) and \(dA_\gamma\) by their Euclidean
counterparts does not change the existence or value of the limit
in \eqref{eq:scalar-flux}.

\begin{proposition}\label{prop:conformal-mass}
For the function \(w\) in Lemma~\ref{lem:scalar-potential},
\(\mathcal F_E(w)\) exists as a finite real number on every end \(E\).
More generally, let \(u\) be a smooth function on \(N\).
Suppose \(u>0\), \(u=1+O_{2,\alpha}(r^{-\sigma})\) on \(E_0\),
and the flux in \eqref{eq:scalar-flux} exists.  Then
\begin{equation}
 m_{E_0}(u^{4/(n-2)}\gamma)
 =m_{E_0}(\gamma)-\frac{2}{(n-2)\omega_{n-1}}\mathcal F_{E_0}(u).
 \label{eq:conformal-mass}
\end{equation}
\end{proposition}

\begin{proof}
Since \(\Delta_\gamma w=\kappa_nR_\gamma w\) and \(0<w\leq1\),
the divergence theorem on end annuli gives
\[
 \left|\int_{S_{r_2}(E)}\partial_{\nu_\gamma}w\,dA_\gamma
 -\int_{S_{r_1}(E)}\partial_{\nu_\gamma}w\,dA_\gamma\right|
 \leq\kappa_n\int_{E\cap\{r>r_1\}}R_\gamma\,dV_\gamma
 \longrightarrow0
\]
as \(r_1\to\infty\), uniformly for \(r_2>r_1\).
Thus \(\mathcal F_E(w)\) exists on every end.

To prove the mass formula for \(u\), write \(p:=4/(n-2)\).
The ADM integrand of
\(\widehat\gamma_{ij}:=u^p\gamma_{ij}\) is
\[
 u^p(\partial_j\gamma_{ij}-\partial_i\gamma_{jj})
 +p u^{p-1}\bigl((\partial_j u)\gamma_{ij}
                   -(\partial_i u)\gamma_{jj}\bigr).
\]
Replacing \(u^p\) by one in the first term and \(\gamma\) by
\(\delta\) in the second gives integrated errors of order
\(O(r^{n-2-\sigma-\tau})=o(1)\).  Replacing \(u^{p-1}\) by one
in the remaining term gives an error
\(O(r^{n-2-2\sigma})=o(1)\).  The conformal contribution that remains
is \(-p(n-1)\partial_{\nu_\delta}u\).  The ADM normalization in
\eqref{eq:adm-mass} now gives \eqref{eq:conformal-mass}.
\end{proof}

\subsection{The mass and energy estimate}

We now combine Proposition~\ref{prop:conformal-mass} with the positive mass
inequality to bound the energy of \(w\).

\begin{proposition}\label{prop:scalar-energy}
Let \(w\) be the function in Lemma~\ref{lem:scalar-potential}, and
write \(\mu:=m_{E_0}(\gamma)\).  The positive mass inequality gives
\begin{equation}
 \int_N\bigl(|dw|_\gamma^2+\kappa_nR_\gamma w^2\bigr)\,dV_\gamma
 =\mathcal F_{E_0}(w)
 \leq\frac{(n-2)\omega_{n-1}}2\mu.
 \label{eq:scalar-energy}
\end{equation}
\end{proposition}

\begin{proof}
Multiplying \(\Delta_\gamma w=\kappa_nR_\gamma w\) by \(w\) and
integrating over an exhaustion yields
\[
 \int_N\bigl(|dw|_\gamma^2+\kappa_nR_\gamma w^2\bigr)\,dV_\gamma
 =\lim_{r\to\infty}\sum_E
   \int_{S_r(E)}w\,\partial_{\nu_\gamma}w\,dA_\gamma.
\]
Every auxiliary end contributes
\(O(r^{n-2-2\sigma})=o(1)\).  On \(E_0\), replacing \(w\) by
one in the boundary integral has the same vanishing error.  This
proves the equality in \eqref{eq:scalar-energy}.

For \(0<t<1\), set
\[
 v_t:=t+(1-t)w,\qquad \gamma_t:=v_t^{4/(n-2)}\gamma.
\]
Since \(t\leq v_t\leq1\), the metric \(\gamma_t\) is complete.
It is AF at \(E_0\), where \(v_t\to1\), and at every auxiliary
end after a constant coordinate dilation, since \(v_t\to t>0\)
there.  Moreover,
\[
 L_\gamma v_t=tR_\gamma,
\]
so the conformal scalar curvature formula gives
\begin{equation}
 R_{\gamma_t}=t v_t^{-(n+2)/(n-2)}R_\gamma\geq0,
 \qquad
 R_{\gamma_t}\,dV_{\gamma_t}=t v_tR_\gamma\,dV_\gamma.
 \label{eq:affine-conformal-curvature}
\end{equation}
Thus the scalar curvature remains integrable, and the positive mass
inequality applies. By Proposition~\ref{prop:conformal-mass},
\(\mathcal F_{E_0}(v_t)=(1-t)\mathcal F_{E_0}(w)\) exists, so
Proposition~\ref{prop:conformal-mass} gives
\[
 0\leq m_{E_0}(\gamma_t)
 =\mu-\frac{2(1-t)}{(n-2)\omega_{n-1}}\mathcal F_{E_0}(w).
\]
Letting \(t\downarrow0\) proves the inequality in
\eqref{eq:scalar-energy}.
\end{proof}

\begin{remark}
The constant in \eqref{eq:scalar-energy} is sharp.  For the spatial
Schwarzschild manifold
\[
 N:=\mathbb R^n\setminus\{0\},\qquad
 \gamma:=\left(1+\frac{\mu}{2r^{n-2}}\right)^{4/(n-2)}\delta,
 \qquad\mu>0,
\]
with the end \(r\to\infty\) distinguished, the required potential is
\(w=(1+\mu/(2r^{n-2}))^{-1}\).  Its outward flux is
\((n-2)\omega_{n-1}\mu/2\), so equality holds.
\end{remark}

\subsection{Localized estimates}

Define
\begin{equation}
 q:=1-w^2.
 \label{eq:scalar-defect}
\end{equation}
Then \(0\leq q<1\), with limit zero on \(E_0\) and limit one on
every auxiliary end, and
\begin{equation}
 -\Delta_\gamma q=2|dw|_\gamma^2+2\kappa_nR_\gamma w^2\geq0.
 \label{eq:scalar-defect-equation}
\end{equation}
The square in \(1-w^2\) ensures that the flux of \(dq=-2w\,dw\)
vanishes on the auxiliary ends.  It therefore permits truncation by a
function that is constant there.

\begin{corollary}\label{cor:scalar-local}
For every \(0<\nu<1\),
\begin{align}
 \int_{\{q<\nu\}}|dq|_\gamma^2\,dV_\gamma
 &\leq(n-2)\omega_{n-1}\nu\mu,
 \label{eq:scalar-truncation}\\
 \int_{\{q<\nu\}}R_\gamma\,dV_\gamma
 &\leq\frac{2(n-1)\omega_{n-1}}{1-\nu}\mu.
 \label{eq:scalar-local-curvature}
\end{align}
\end{corollary}

\begin{proof}
Multiply \eqref{eq:scalar-defect-equation} by the Lipschitz function
\(\min\{q,\nu\}\), integrate over the compact domain obtained by
truncating every end at \(S_r\), and integrate by parts. The identity
\(d\min\{q,\nu\}=\mathbf 1_{\{q<\nu\}}\,dq\) holds almost everywhere.
We then let \(r\to\infty\). On \(E_0\),
\(q=O(r^{-\sigma})\) and \(dq=O(r^{-\sigma-1})\), so the boundary
term is \(O(r^{n-2-2\sigma})=o(1)\).  On an auxiliary end the test
function eventually equals \(\nu\), while
\(dq=-2w\,dw=O(r^{-2\sigma-1})\), so its boundary term has the same
vanishing order.  It follows that
\begin{align*}
 \int_{\{q<\nu\}}|dq|_\gamma^2\,dV_\gamma
 &=2\int_N\min\{q,\nu\}
       \bigl(|dw|_\gamma^2+\kappa_nR_\gamma w^2\bigr)\,dV_\gamma\\
 &\leq2\nu\int_N
       \bigl(|dw|_\gamma^2+\kappa_nR_\gamma w^2\bigr)\,dV_\gamma.
\end{align*}
Equation~\eqref{eq:scalar-energy} proves
\eqref{eq:scalar-truncation}.  On \(\{q<\nu\}\) one has
\(w^2>1-\nu\), and hence
\[
 \kappa_n(1-\nu)\int_{\{q<\nu\}}R_\gamma\,dV_\gamma
 \leq\kappa_n\int_N R_\gamma w^2\,dV_\gamma
 \leq\frac{(n-2)\omega_{n-1}}2\mu.
\]
Substitution of \(\kappa_n=(n-2)/(4(n-1))\) gives
\eqref{eq:scalar-local-curvature}.
\end{proof}

\section{The combined coordinate defect}\label{sec:combined-defect}

In this section, we construct the coordinate map \(\Phi\) from the
minimal graph functions and define a nonnegative function \(Q\) that
combines the deviation of their normalized Gram matrix from \(I\) with
the scalar defects \(q_a=1-w_a^2\). Where \(Q\) is small, \(\Phi\) has
small metric distortion and the tilted metrics are uniformly comparable
to \(g\). On these regions, the scalar curvature identity and the
mass--energy estimates give integral bounds for the Hessians of the graph
functions. Together with the truncated estimates for \(dq_a\), these
bounds yield \eqref{eq:good-band-energy}, the main estimate of
Proposition~\ref{prop:good-band}. Since \(Q\) tends to zero on the
distinguished end and to \(2n\) on the auxiliary ends, this estimate
will allow us to select a level set of small area and extract an exterior
region with global coordinates in Section~\ref{sec:regular-exterior}.

Assume the hypotheses of Theorem~\ref{thm:quantitative}.  Use the
slope \(\lambda=\lambda_0(n)/2\) fixed above, write
\(m:=m(g)\geq0\), and set \(u_a:=u_{\lambda e_a}\) for
\(1\leq a\leq n\).  Recall that
\[
 \eta_a:=(1+|du_a|_g^2)^{-1/2},\qquad
 \bar g_a:=g+du_a\otimes du_a,\qquad
 \widetilde g_a:=\eta_a^{2/(n-1)}\bar g_a.
\]
Section~\ref{sec:graph-geometry} verifies that each \(\widetilde g_a\)
is complete and AF, with nonnegative integrable scalar curvature.
Normalize its limiting metric on each end by a linear change of
coordinates, and let \(w_a\) be the function given by
Lemma~\ref{lem:scalar-potential}, with the same distinguished end.
Define
\begin{equation}
\begin{aligned}
 \Phi:=\left(\frac{u_1}{\lambda},\ldots,
                   \frac{u_n}{\lambda}\right),\qquad
 &G_{ab}:=\lambda^{-2}\langle du_a,du_b\rangle_g,\\
 q_a:=1-w_a^2,\qquad
 &Q:=|G-I|^2+\sum_{a=1}^n q_a.
\end{aligned}
 \label{eq:combined-defect}
\end{equation}
Here \(|\cdot|\) is the Euclidean Hilbert--Schmidt matrix norm.
The graph expansions in Theorem~\ref{thm:affine-graph} give
\[
 \Phi(x)=x+o(r),\qquad d\Phi=\operatorname{Id}+o(1),\qquad
 G\longrightarrow I
\]
on the distinguished end.  Thus \(Q\to0\) there.  On every auxiliary
end, \(du_a\to0\) and \(w_a\to0\), so
\begin{equation}
 Q\longrightarrow |I|^2+n=2n.
 \label{eq:defect-auxiliary-limit}
\end{equation}
In particular, small sublevel sets of \(Q\) avoid the auxiliary ends.

\begin{proposition}\label{prop:good-band}
There are constants \(\epsilon_0(n)>0\) and \(C(n)<\infty\) such
that, for \(0<\epsilon<\epsilon_0(n)\),
\begin{equation}
 \int_{\{Q<6\epsilon\}}|dQ|_g^2\,dV_g\leq C(n)\epsilon m.
 \label{eq:good-band-energy}
\end{equation}
\end{proposition}

\begin{proof}
Write \(B_\epsilon:=\{Q<6\epsilon\}\).  On this set,
\[
 |G-I|<\sqrt{6\epsilon},\qquad |du_a|_g^2=\lambda^2G_{aa}.
\]
For sufficiently small \(\epsilon_0(n)\), these inequalities bound
\(|du_a|_g\) from above and \(\eta_a\) away from zero by dimensional
constants.  Consequently,
\begin{equation}
 C(n)^{-1}g\leq \widetilde g_a\leq \bar g_a\leq C(n)g
 \qquad\text{on }B_\epsilon.
 \label{eq:good-band-comparison}
\end{equation}
The constants are dimensional because \(\lambda\) has already been
fixed in terms of \(n\).

Since \(|G-I|^2\geq0\) and \(q_a=1-w_a^2\geq0\) for every \(a\),
the definition \eqref{eq:combined-defect} gives \(q_a\leq Q\).
Hence \(B_\epsilon\subset\{q_a<6\epsilon\}\) for every \(a\). Take
\(\epsilon_0<1/12\).  Equations~\eqref{eq:scalar-local-curvature}
and \eqref{eq:tilted-mass-formula} give
\begin{equation}
 \int_{B_\epsilon}R_{\widetilde g_a}\,dV_{\widetilde g_a}
 \leq C(n)m(\widetilde g_a)\leq C(n)m.
 \label{eq:good-band-scalar}
\end{equation}
On the other hand, \eqref{eq:3.5}--\eqref{eq:3.6} yield the
density identity
\begin{equation}
 R_{\widetilde g_a}\,dV_{\widetilde g_a}
 =\eta_a^{-1/(n-1)}
 \left(R_g+|\eta_a\nabla^2u_a|_{\bar g_a}^2
       +\frac n{n-1}|d\log\eta_a|_{\bar g_a}^2\right)dV_g.
 \label{eq:tilted-scalar-density}
\end{equation}
All three terms on the right are nonnegative.  The metric comparison
in \eqref{eq:good-band-comparison} and the lower bound for \(\eta_a\)
therefore imply
\begin{equation}
 \sum_{a=1}^n\int_{B_\epsilon}|\nabla^2u_a|_g^2\,dV_g
 \leq C(n)m.
 \label{eq:good-band-hessians}
\end{equation}

Differentiating \(G_{ab}=\lambda^{-2}\langle du_a,du_b\rangle_g\)
and using the gradient bounds gives
\[
 |dG|_g^2\leq C(n)\sum_{a=1}^n|\nabla^2u_a|_g^2
 \qquad\text{on }B_\epsilon.
\]
Since \(|G-I|^2<6\epsilon\) there,
\[
 \bigl|d(|G-I|^2)\bigr|_g^2
 \leq4|G-I|^2|dG|_g^2
 \leq C(n)\epsilon\sum_{a=1}^n|\nabla^2u_a|_g^2.
\]
After integration, \eqref{eq:good-band-hessians} gives
\begin{equation}
 \int_{B_\epsilon}\bigl|d(|G-I|^2)\bigr|_g^2\,dV_g
 \leq C(n)\epsilon m.
 \label{eq:good-band-gram-energy}
\end{equation}
The other terms are controlled by the truncated scalar estimate:
\begin{align}
 \int_{B_\epsilon}|dq_a|_g^2\,dV_g
 &\leq C(n)\int_{B_\epsilon}
        |dq_a|_{\widetilde g_a}^2\,dV_{\widetilde g_a}\notag\\
 &\leq C(n)\int_{\{q_a<6\epsilon\}}
        |dq_a|_{\widetilde g_a}^2\,dV_{\widetilde g_a}
 \leq C(n)\epsilon m,
 \label{eq:good-band-scalar-energy}
\end{align}
where we used \eqref{eq:good-band-comparison},
\eqref{eq:scalar-truncation}, and the tilted mass formula
\eqref{eq:tilted-mass-formula}.
Finally,
\[
 |dQ|_g^2\leq(n+1)\left(
      \bigl|d(|G-I|^2)\bigr|_g^2+\sum_{a=1}^n|dq_a|_g^2\right).
\]
Combining \eqref{eq:good-band-gram-energy} and
\eqref{eq:good-band-scalar-energy} proves
\eqref{eq:good-band-energy}.
\end{proof}

\begin{remark}
When \(M\) is spin, one can also use Witten spinors for the tilted
metrics and the Lichnerowicz identity to obtain an alternative control
of the coordinate defect
\cite{Witten1981,ParkerTaubes1982}.
Related spinorial estimates appear in
\cite{BrayFinster2002,FinsterKath2002,Finster2009}.
\end{remark}

\section{Extraction of the regular exterior}
\label{sec:regular-exterior}

Recall that \(\Phi=(u_1/\lambda,\ldots,u_n/\lambda)\) is the map
constructed above and \(Q=|G-I|^2+\sum_aq_a\) is its defect. Where
\(Q\) is small, \(\Phi\) is a local diffeomorphism with small metric
distortion. The scalar potentials ensure that this region avoids the
auxiliary ends. We use the energy bound for \(dQ\), an isoperimetric
estimate, and coarea to select a smooth level set \(\{Q=t\}\) of small
area as a cutting boundary. Its image under \(\Phi\) may self-intersect,
so the following lemma provides a smooth neighborhood with controlled
boundary area. The degree at infinity then yields global coordinates
on a smaller exterior region.

\begin{lemma}\label{lem:wall-smoothing}
Let \(n\geq2\), let \((\Sigma^{n-1},h)\) be a compact, nonempty,
smooth manifold without boundary, and let
\(F:(\Sigma,h)\to(\mathbb R^n,\delta)\) be a smooth immersion.
Set \(\Gamma:=F(\Sigma)\).  For every \(\varrho>0\), there is a
bounded open set \(\mathcal O_\varrho\) with smooth boundary such that
\[
 \Gamma\subset\mathcal O_\varrho
 \Subset\{x:\operatorname{dist}_\delta(x,\Gamma)<\varrho\}
\]
and
\begin{equation}
 \mathcal H_\delta^{n-1}(\partial\mathcal O_\varrho)
 \leq C_n\int_\Sigma J_{n-1}F\,dA_h,
 \label{eq:wall-smoothing}
\end{equation}
where \(J_{n-1}F\) is the tangential Jacobian.  The closure
\(D_\varrho\) of the unbounded component of
\(\mathbb R^n\setminus\overline{\mathcal O_\varrho}\) is a smooth
Euclidean exterior region disjoint from \(\Gamma\), and its boundary
satisfies the same area bound.
\end{lemma}

\begin{proof}
Write \(N_s(\Gamma):=\{x:\operatorname{dist}_\delta(x,\Gamma)<s\}\)
and let \(P\) denote Euclidean perimeter.  By the definition of
Hausdorff measure and compactness, \(\Gamma\) has a finite cover by
balls \(B_{r_j}(a_j)\), with \(a_j\in\Gamma\) and
\(r_j<\varrho/4\), such that
\[
 \sum_jr_j^{n-1}\leq C_n\mathcal H_\delta^{n-1}(\Gamma)
 \leq C_n\int_\Sigma J_{n-1}F\,dA_h.
\]
Set \(U:=\bigcup_j B_{r_j}(a_j)\). Then
\(\Gamma\Subset U\Subset N_{\varrho/2}(\Gamma)\) and
\[
 P(U)\leq\sum_jP(B_{r_j}(a_j))\leq C_n\sum_jr_j^{n-1}.
\]
Convolve \(\mathbf1_U\) with a nonnegative smooth mollifier of
sufficiently small support.  The resulting function \(f\) equals one
near \(\Gamma\), has compact support in \(N_\varrho(\Gamma)\), and
satisfies \(\int_{\mathbb R^n}|Df|\leq P(U)\).  Coarea and Sard's
theorem give a regular value \(a\in(0,1)\) with
\(\mathcal H_\delta^{n-1}(\{f=a\})\leq2P(U)\).
The set \(\mathcal O_\varrho:=\{f>a\}\) has the required properties.
Finally, \(\partial D_\varrho\) is a union of components of
\(\partial\mathcal O_\varrho\), so its area does not increase.
\end{proof}

The next lemma extends the level selection argument of
\cite[Lemmas~3.2 and~3.3]{DongSong2025} to every dimension \(n\geq3\).
We prove the required volume estimate directly from the distribution
of \(Q\).  

\begin{lemma}\label{lem:small-level}
Assume the hypotheses and notation of
Proposition~\ref{prop:good-band}.  Let \(0<\epsilon<\epsilon_0(n)\)
and let \(\mathcal U\) be the component of \(\{Q<6\epsilon\}\)
containing the distinguished end.  Set
\[
 \mathcal E_Q:=\int_{\mathcal U}|dQ|_g^2\,dV_g,
 \qquad S_t:=\mathcal U\cap\{Q=t\}.
\]
If \(\mathcal E_Q>0\), there is a regular value
\(t_\epsilon\in(0,6\epsilon)\) such that
\[
 \mathcal H_g^{n-1}(S_{t_\epsilon})
 \leq C_n
 \left(\frac{\mathcal E_Q}{\epsilon^2}\right)^{\frac{n-1}{n-2}}.
\]
\end{lemma}

\begin{proof}
The eigenvalues of \(G\) on \(\mathcal U\) lie in
\([1-\sqrt{6\epsilon},1+\sqrt{6\epsilon}]\).  Thus the singular
values of \(d\Phi\) are bounded above and below by dimensional
constants.  Its orientation sign is positive near infinity and hence
throughout \(\mathcal U\).  Also, \(Q\to0\) on the distinguished
end and \(Q\to2n>6\epsilon\) on every auxiliary end, after
reducing \(\epsilon_0(n)\).  Consequently, each \(S_t\) with
\(0<t<6\epsilon\) is compact.

For regular \(0<s<t<5\epsilon\), write
\(\mathcal A_{s,t}:=\mathcal U\cap\{s<Q<t\}\).  This slab has
compact closure and boundary \(S_s\cup S_t\).  Although \(\Phi\)
need not be injective, orientation preservation gives
\[
 \Phi_\#[\![\mathcal A_{s,t}]\!]
 =\chi_{s,t}[\![\mathbb R^n]\!],\qquad
 \chi_{s,t}(y):=\#\bigl(\Phi^{-1}(y)\cap\mathcal A_{s,t}\bigr).
\]
Here \([\![\cdot]\!]\) denotes the integration current with the
chosen orientation.  The multiplicity \(\chi_{s,t}\) is compactly
supported, nonnegative, and integer valued.  The area formula and
commutation of pushforward with boundary imply
\[
 |D\chi_{s,t}|(\mathbb R^n)
 \leq C_n\bigl(\mathcal H_g^{n-1}(S_s)
                    +\mathcal H_g^{n-1}(S_t)\bigr).
\]
Since \(\chi_{s,t}^{n/(n-1)}\geq\chi_{s,t}\), the area formula and
the Euclidean \(BV\) Sobolev inequality give
\begin{equation}
 \begin{aligned}
 \operatorname{Vol}_g(\mathcal A_{s,t})^{(n-1)/n}
 &\leq C_n\|\chi_{s,t}\|_{L^{n/(n-1)}(\mathbb R^n)}\\
 &\leq C_n|D\chi_{s,t}|(\mathbb R^n)\\
 &\leq C_n\bigl(\mathcal H_g^{n-1}(S_s)
                    +\mathcal H_g^{n-1}(S_t)\bigr).
 \end{aligned}
 \label{eq:level-bv-isoperimetric}
\end{equation}
See \cite[Theorem~5.10(i)]{EvansGariepy2015} for the Sobolev
inequality.  Positivity of the multiplicity is essential here: it
prevents cancellation between different sheets.

If the infimum of \(\mathcal H_g^{n-1}(S_t)\) over regular
\(t\in(0,5\epsilon)\) is zero, the conclusion follows immediately.
Otherwise choose a regular \(t_0\in(0,5\epsilon)\) such that
\[
 \mathcal H_g^{n-1}(S_{t_0})
 \leq2\inf_{\substack{0<t<5\epsilon\\
 t\ {\rm regular}}}
                   \mathcal H_g^{n-1}(S_t).
\]
Choose an interval of level values with endpoint \(t_0\) and length
at least \(2\epsilon\) as follows. Set
\[
 (\sigma,L):=
 \begin{cases}
 (-1,t_0-\epsilon),&t_0\geq3\epsilon,\\
 (1,5\epsilon-t_0),&t_0<3\epsilon,
 \end{cases}
\]
so \(L\geq2\epsilon\).  For \(0<h<L\), let \(\mathcal B_h\)
be the slab between the levels \(t_0\) and \(t_0+\sigma h\), and set
\[
 W(h):=\operatorname{Vol}_g
       (\mathcal B_h\cap\{|dQ|_g\ne0\}),\qquad
 A(h):=\mathcal H_g^{n-1}(S_{t_0+\sigma h}).
\]
By the coarea formula, the volume \(W(h)\) of the noncritical part of
the slab is absolutely continuous in \(h\), with derivative given
almost everywhere by
\[
 W'(h)=\int_{S_{t_0+\sigma h}}\frac1{|dQ|_g}\,dA_g>0.
\]
The positivity of \(W'(h)\) follows because every regular level in
\((0,5\epsilon)\) has positive area.  By
\eqref{eq:level-bv-isoperimetric} and the choice of \(t_0\),
\[
 W^{(n-1)/n}(h)\leq C_nA(h).
\]
Surface Cauchy--Schwarz then gives, with
\(\beta:=2(n-1)/n>1\),
\[
 \frac{W^\beta(h)}{W'(h)}
 \leq C_n\frac{A^2(h)}{W'(h)}
 \leq C_n\int_{S_{t_0+\sigma h}}|dQ|_g\,dA_g.
\]
Integrating and applying coarea once more yields
\begin{equation}
 \int_0^L\frac{W^\beta(h)}{W'(h)}\,dh
 \leq C_n\mathcal E_Q.
 \label{eq:radial-rearrangement-energy}
\end{equation}

If \(W(\epsilon)=0\), the estimate below is immediate.  Otherwise
\(W\geq W(\epsilon)>0\) on \([\epsilon,L]\), so Cauchy--Schwarz
and the chain rule for absolutely continuous functions give
\[
 \begin{aligned}
 (L-\epsilon)^2
 &\leq
 \left(\int_\epsilon^L\frac{W^\beta}{W'}\,dh\right)
 \left(\int_\epsilon^L\frac{W'}{W^\beta}\,dh\right)\\
 &\leq C_n\mathcal E_Q\,
 \frac{W^{1-\beta}(\epsilon)-W^{1-\beta}(L)}{\beta-1}\\
 &\leq C_n\mathcal E_Q\,W^{1-\beta}(\epsilon).
 \end{aligned}
\]
Since \(L-\epsilon\geq\epsilon\) and
\(\beta-1=(n-2)/n\), we obtain
\begin{equation}
 \operatorname{Vol}_g
 (\mathcal B_\epsilon\cap\{|dQ|_g\ne0\})
 =W(\epsilon)
 \leq C_n
 \left(\frac{\mathcal E_Q}{\epsilon^2}\right)^{n/(n-2)}.
 \label{eq:capacity-band-volume}
\end{equation}

Let \(I\) be the interval between \(t_0\) and
\(t_0+\sigma\epsilon\), whose length is \(\epsilon\).
Finally, coarea and Cauchy--Schwarz imply
\[
 \begin{aligned}
 \int_I\mathcal H_g^{n-1}(S_t)\,dt
 &=\int_{\mathcal B_\epsilon}|dQ|_g\,dV_g\\
 &\leq\mathcal E_Q^{1/2}W^{1/2}(\epsilon).
 \end{aligned}
\]
Divide by \(\epsilon\), use \eqref{eq:capacity-band-volume}, and
choose a regular value by Sard's theorem.  This proves the lemma.
\end{proof}

The selected level bounds a region on which \(\Phi\) is locally
invertible.  Its boundary image may enclose several sheets.  The next
lemma retains the unique sheet over the unbounded component of the
complement of that image.

\begin{lemma}\label{lem:one-sheet-extraction}
Assume the hypotheses and notation of
Proposition~\ref{prop:good-band}.  Let
\(0<\epsilon<\epsilon_0(n)\), let \(\mathcal U\) and
\(\mathcal E_Q>0\) be as in Lemma~\ref{lem:small-level}, and let
\(t_\epsilon\) be the regular value obtained there.  If
\(\mathcal C\) is the component of \(\{Q<t_\epsilon\}\) containing
the distinguished end, then there are exterior regions
\(E\subset\overline{\mathcal C}\) and
\(\Omega'\subset\mathbb R^n\) such that
\(\Phi:E\to\Omega'\) is an orientation-preserving diffeomorphism and
\[
 \mathcal H_g^{n-1}(\partial E)
 \leq C_n\mathcal H_g^{n-1}(S_{t_\epsilon}).
\]
\end{lemma}

\begin{proof}
The boundary \(\partial\mathcal C\) is a compact smooth union of
components of \(S_{t_\epsilon}\).  In particular,
\[
 \mathcal H_g^{n-1}(\partial\mathcal C)
 \leq\mathcal H_g^{n-1}(S_{t_\epsilon}).
\]
The closure of \(\mathcal C\) meets no auxiliary end outside a compact
set.  Since \(\Phi=x+o(r)\) on the distinguished end,
\(\Phi:\overline{\mathcal C}\to\mathbb R^n\) is proper.  Put
\(\Gamma:=\Phi(\partial\mathcal C)\), and let \(D_\infty\) be the
unbounded component of \(\mathbb R^n\setminus\Gamma\).
The number of preimages in \(\mathcal C\) is finite and locally
constant on \(D_\infty\), by properness and local invertibility.

We compute this number at a point \(y\) outside the image of a large
compact truncation of \(\overline{\mathcal C}\).  Truncate again at
a much larger coordinate sphere.  The inner boundary image lies in a
ball avoiding \(y\), so its contribution to the degree is zero.
On the outer sphere, the straight homotopy from \(\Phi(x)\) to \(x\)
avoids \(y\), because \(\Phi(x)=x+o(r)\).  Its degree is therefore
one.  Every preimage has positive local degree, so \(y\) has exactly
one preimage.  Local constancy now gives the same conclusion throughout
\(D_\infty\).  Thus
\[
 \Phi:\Phi^{-1}(D_\infty)\cap\mathcal C\longrightarrow D_\infty
\]
is a diffeomorphism.

If \(\partial\mathcal C=\varnothing\), connectedness gives
\(\mathcal C=M\). In this case, the preceding argument yields a global
diffeomorphism onto \(\mathbb R^n\), and we take
\(E:=M\), \(\Omega':=\mathbb R^n\).  Otherwise apply
Lemma~\ref{lem:wall-smoothing} to
\(\Phi|_{\partial\mathcal C}\), and let \(\Omega'\) be the closure
of the unbounded component of
\(\mathbb R^n\setminus\overline{\mathcal O_\varrho}\).
Because \(\Gamma\subset\mathcal O_\varrho\), this closure lies in
\(D_\infty\), including its compact smooth boundary.  Define
\(E\) using the unique inverse of \(\Phi\) on \(\Omega'\).
It is a smooth exterior region containing the distinguished end.
It is closed in \(M\) because \(\Omega'\) is closed and disjoint from
\(\Phi(\partial\mathcal C)\).
The bounds for the singular values of \(d\Phi\), the area formula,
and \eqref{eq:wall-smoothing} give
\[
 \begin{aligned}
 \mathcal H_g^{n-1}(\partial E)
 &\leq C_n\mathcal H_\delta^{n-1}(\partial\Omega')\\
 &\leq C_n\int_{\partial\mathcal C}
         J_{n-1}(\Phi|_{\partial\mathcal C})\,dA_g\\
 &\leq C_n\mathcal H_g^{n-1}(\partial\mathcal C).
 \end{aligned}
\]
This proves the claim.
\end{proof}

\begin{proof}[Proof of Theorem~\ref{thm:quantitative}]
Suppose that \(m>0\).
Fix \(0<\epsilon<\epsilon_0(n)\), and let \(\mathcal U\) be the
component of \(\{Q<6\epsilon\}\) containing the distinguished end.
Proposition~\ref{prop:good-band} gives
\[
 \mathcal E_Q:=\int_{\mathcal U}|dQ|_g^2\,dV_g
 \leq C(n)\epsilon m.
\]
If \(\mathcal E_Q=0\), then \(Q=0\) on \(\mathcal U\), since it
is constant there and tends to zero at infinity.  Continuity makes
\(\mathcal U\) closed as well as open, so \(\mathcal U=M\).
There are no auxiliary ends, and \(\Phi\) is a proper local
diffeomorphism of degree one.  Thus \(\Phi:M\to\mathbb R^n\) is a
diffeomorphism.  Since \(G=I\), it is an isometry, and we take
\(E_\epsilon:=M\), \(\Omega_\epsilon:=\mathbb R^n\), and
\(\Phi_\epsilon:=\Phi\).

If \(\mathcal E_Q>0\), Lemmas~\ref{lem:small-level} and
\ref{lem:one-sheet-extraction} give exterior regions
\(E_\epsilon\), \(\Omega_\epsilon\) and a diffeomorphism
\(\Phi_\epsilon:=\Phi|_{E_\epsilon}\) between them, with
\begin{equation}
 \mathcal H_g^{n-1}(\partial E_\epsilon)
 \leq C(n)
 \left(\frac{m}{\epsilon}\right)^{\frac{n-1}{n-2}}.
 \label{eq:regular-boundary}
\end{equation}
Let \(g_\epsilon^\Phi:=(\Phi_\epsilon^{-1})^*g\).
Its inverse matrix is \(G\), because
\(\Phi_\epsilon^*dy^a=d\Phi^a\).  Matrix inversion and
\(|G-I|\leq\sqrt{6\epsilon}<1/2\) therefore give
\begin{equation}
 \sup_{\Omega_\epsilon}|g_\epsilon^\Phi-\delta|_\delta
 \leq C(n)\sqrt\epsilon.
 \label{eq:regular-metric}
\end{equation}
The required asymptotics follow from those of \(\Phi\).
\end{proof}
\section{Intrinsic distances and measured convergence}
\label{sec:inner-distance}

The preceding metric comparison controls the lengths of curves
contained in the exterior region. Convergence of the intrinsic
distances requires an additional argument because curves must avoid
the excised set.  Dong's Euclidean
excision theorem removes this obstruction.  We then transfer the
conclusion to the Riemannian metric.  The same comparison also allows
us to remove the assumption of harmonic asymptotics.

\subsection{The case of harmonic asymptotics}

\begin{lemma}
\label{lem:bilinear-transfer}
Let \(Z_i\) be connected smooth \(n\)-manifolds, possibly with boundary,
with Riemannian metrics \(g_{0i}\) and \(g_{1i}\).  Assume that
\((Z_i,\widehat d_{g_{0i}})\) is complete and
\[
 (1-\vartheta_i)g_{0i}\leq g_{1i}
 \leq(1+\vartheta_i)g_{0i},
 \qquad 0\leq\vartheta_i<1,\qquad \vartheta_i\to0.
\]
Then \((Z_i,\widehat d_{g_{1i}})\) is complete, and
\[
 \sqrt{1-\vartheta_i}\,\widehat d_{g_{0i}}
 \leq\widehat d_{g_{1i}}
 \leq\sqrt{1+\vartheta_i}\,\widehat d_{g_{0i}},
\]
\[
 (1-\vartheta_i)^{n/2}dV_{g_{0i}}
 \leq dV_{g_{1i}}
 \leq(1+\vartheta_i)^{n/2}dV_{g_{0i}},
\]
and
\[
 \begin{aligned}
 (1-\vartheta_i)^{(n-1)/2}
 \mathcal H_{g_{0i}}^{n-1}(\partial Z_i)
 &\leq \mathcal H_{g_{1i}}^{n-1}(\partial Z_i)\\
 &\leq(1+\vartheta_i)^{(n-1)/2}
 \mathcal H_{g_{0i}}^{n-1}(\partial Z_i).
 \end{aligned}
\]
For any \(z_i\in Z_i\), the convergence
\[
 (Z_i,\widehat d_{g_{0i}},z_i,dV_{g_{0i}})
 \xrightarrow{\mathrm{pmGH}}(X,d,z,\mu)
\]
implies the same convergence with \(g_{0i}\) replaced by \(g_{1i}\).
\end{lemma}

\begin{proof}
Comparison of curve lengths gives the distance bounds and transfers
completeness.  The volume and boundary estimates follow by
diagonalizing \(g_{1i}\) relative to \(g_{0i}\).  Both intrinsic spaces
are proper, since they are complete, locally compact length spaces.

Fix a finite radius \(R\).  The distance bounds place the two
\(R\)-balls in a fixed, slightly larger \(g_{0i}\)-ball.  On that ball,
\[
 |\widehat d_{g_{1i}}-\widehat d_{g_{0i}}|
 \leq C_R\vartheta_i.
\]
With \(\rho_i:=dV_{g_{1i}}/dV_{g_{0i}}\), we also have
\(\|\rho_i-1\|_{L^\infty}\leq C_n\vartheta_i\).
The assumed measured convergence bounds the \(g_{0i}\)-volume of the
enlarged ball, so the two measures differ there by \(o(1)\) in total
variation.  Composing the given approximations with the identity map
therefore gives the same measured limit.  Applying this argument at
radii for which \(\mu(\partial B(z,R))=0\) proves the pointed assertion.
\end{proof}

Dong--Song proved the required Euclidean excision result in dimension
three \cite[Theorem~4.1]{DongSong2025}.  Its extension to every
\(n\geq2\), including measured convergence for arbitrary base points,
is \cite[Appendix~A, Theorem~A.1]{Dong2025}.

\begin{proof}[Proof of Theorem~\ref{thm:main} under harmonic asymptotics]
Assume that each distinguished end has harmonic asymptotics.  Theorem~\ref{thm:positive-mass} gives \(m_i\geq0\).  For all sufficiently
large indices with \(m_i>0\), apply
Theorem~\ref{thm:quantitative} with \(\epsilon_i:=\sqrt{m_i}\).
We obtain smooth connected exterior regions \(E_i\) and
diffeomorphisms \(\Phi_i:E_i\to\Omega_i\) satisfying
\[
 \begin{aligned}
 \mathcal H_{g_i}^{n-1}(\partial E_i)
 &\leq C(n)m_i^{\frac{n-1}{2(n-2)}}\longrightarrow0,\\
 \| (\Phi_i^{-1})^*g_i-\delta\|_{C^0(\Omega_i)}
 &\leq C(n)m_i^{1/4}\longrightarrow0.
 \end{aligned}
\]
The boundary measure comparison gives
\(\mathcal H_\delta^{n-1}(\partial\Omega_i)\to0\).

Apply \cite[Appendix~A, Theorem~A.1]{Dong2025} to
\(\operatorname{int}\Omega_i\).  It gives smooth closed subregions
with boundary area tending to zero and the Euclidean pointed measured
limit for every choice of base points.  The construction leaves the
region unchanged outside a compact set.  Retain the component
containing infinity and denote it by \(\Omega_i'\).  This only discards
boundary components and does not change intrinsic pointed balls based
in \(\Omega_i'\).  Thus \(\Omega_i'\) is a complete smooth exterior
region and, for every \(y_i\in\Omega_i'\),
\[
 \mathcal H_\delta^{n-1}(\partial\Omega_i')\longrightarrow0,
 \qquad
 (\Omega_i',\widehat d_{\mathrm{Eucl}},y_i,dx)
 \xrightarrow{\mathrm{pmGH}}
 (\mathbb R^n,d_{\mathrm{Eucl}},0,dx).
\]
Here \(dx\) is restricted Lebesgue measure. Indeed, the Euclidean
isoperimetric inequality gives
\[
 \operatorname{Vol}_\delta(\mathbb R^n\setminus\Omega_i')
 \leq C_n\mathcal H_\delta^{n-1}(\partial\Omega_i')^{n/(n-1)}
 \longrightarrow0.
\]
This estimate is invariant under translations of the base point, and
the limiting measure is Lebesgue measure without normalization.

Set \(\widehat E_i:=\Phi_i^{-1}(\Omega_i')\).
Lemma~\ref{lem:bilinear-transfer}, with
\(g_{0i}:=\delta\) and \(g_{1i}:=(\Phi_i^{-1})^*g_i\), gives
\[
 \mathcal H_{g_i}^{n-1}(\partial\widehat E_i)\longrightarrow0
\]
and, for every \(x_i\in\widehat E_i\),
\[
 (\widehat E_i,\widehat d_{g_i},x_i,dV_{g_i})
 \xrightarrow{\mathrm{pmGH}}
 (\mathbb R^n,d_{\mathrm{Eucl}},0,dx).
\]
If \(m_i=0\), then \((M_i,g_i)\cong\mathbb R^n\), and we take
\(\widehat E_i:=M_i\).  Choose arbitrary smooth coordinate exterior
regions for the finitely many remaining initial indices.  These choices
do not affect the limit.
\end{proof}

\subsection{General asymptotics}
\label{subsec:density}

Approximation by metrics with harmonic asymptotics and convergent ADM
masses, together with global bilinear comparison, extends the result
to general AF metrics. The approximation must preserve asymptotic
flatness on all auxiliary ends for the graph and scalar constructions
to apply.

\begin{lemma}
\label{lem:finite-end-density}
Let \((M^n,g)\), \(n\geq3\), be a smooth, connected, complete oriented
manifold without boundary, with finitely many AF ends,
\(R_g\geq0\), and \(R_g\in L^1(M)\).  Fix a distinguished end.
For every \(0<\beta<1/2\), there is a smooth metric \(\check g\)
such that
\[
 R_{\check g}\geq0,\qquad R_{\check g}\in L^1(M),\qquad
 |m(\check g)-m(g)|<\beta,
\]
\[
 (1-\beta)g\leq\check g\leq(1+\beta)g.
\]
The distinguished end has harmonic asymptotics for \(\check g\).
Every auxiliary end remains asymptotically flat after a constant
dilation of its coordinates.
\end{lemma}

\begin{proof}
Let \(\tau\) be the least AF rate of the ends.  Choose \(p_*>n\)
and
\[
 \frac{n-2}{2}<q'<q<\min\{\tau,n-2\}.
\]
The weighted \(C^{2,\alpha}_{-\tau}\) hypothesis implies the weighted
Sobolev hypotheses of Lee--Lesourd--Unger
\cite[Theorem~1.3]{LeeLesourdUnger2023}, with exponent \(p_*\) and
weight \(q\). Apply that theorem with approximation weight \(q'\)
and a compact set containing the core and the inner portions of all
ends. The theorem yields the mass and bilinear estimates, nonnegative scalar
curvature, and harmonic flatness on the distinguished tail.  Since
\(g\) and the cutoffs are smooth, elliptic regularity makes the
approximating metric smooth.

On the auxiliary ends, the construction in
\cite[Section~3, equation~(3.1)]{LeeLesourdUnger2023} uses a preliminary
metric \(g_{\rm pre}\) that equals \(g\) outside the distinguished tail, and
\[
 \check g=\phi^{4/(n-2)}g_{\rm pre},
\]
where \(\phi\) is positive and uniformly bounded above and away from
zero.  The cutoff equals one on each auxiliary end, so the potential
in the conformal equation vanishes there and \(\Delta_g\phi=0\).
Boundedness and annular elliptic estimates place \(\phi\) in weighted
Sobolev spaces of every small positive weight.  The exterior harmonic
expansion \cite[Theorem~1.17]{Bartnik1986}, followed by Schauder
estimates, gives
\[
 \phi=c+O_{2,\alpha}(r^{-\gamma})
 \quad\text{for every }0<\gamma<n-2,
 \qquad c>0.
\]
Indeed, the leading bounded harmonic mode is constant. After subtracting
it, there are no further homogeneous harmonic degrees between
\(2-n\) and zero.  Iterating the weighted estimate therefore gives
every decay rate strictly below \(n-2\).
In the dilated coordinates \(y:=c^{2/(n-2)}x\),
\[
 \check g_{ij}(y)-\delta_{ij}
 =O_{2,\alpha}\bigl(|y|^{-\min\{\tau,\gamma\}}\bigr).
\]
Choosing \(\gamma>(n-2)/2\) proves asymptotic flatness.

On each auxiliary end, the conformal formulas give
\[
 R_{\check g}=\phi^{-4/(n-2)}R_g,
 \qquad
 |R_{\check g}|\,dV_{\check g}=\phi^2|R_g|\,dV_g.
\]
Hence scalar curvature remains integrable there.  It is integrable on
the distinguished end by the density theorem and on the compact core
by smoothness.  This proves the lemma.
\end{proof}

\begin{proof}[Proof of Theorem~\ref{thm:main}]
Choose \(0<\beta_i<1/3\) with \(\beta_i\downarrow0\), and apply
Lemma~\ref{lem:finite-end-density} to obtain \(\check g_i\) satisfying
\[
 |m(\check g_i)-m_i|<\beta_i,
 \qquad
 (1-\beta_i)g_i\leq\check g_i\leq(1+\beta_i)g_i.
\]
The metrics \(\check g_i\) are complete.  Theorem~\ref{thm:positive-mass} gives \(0\leq m(\check g_i)\leq m_i+\beta_i\to0\).
The result for harmonic asymptotics therefore yields smooth connected
exterior regions \(\widehat E_i\) with
\[
 \mathcal H_{\check g_i}^{n-1}(\partial\widehat E_i)\longrightarrow0
\]
and, for every \(x_i\in\widehat E_i\),
\[
 (\widehat E_i,\widehat d_{\check g_i},x_i,dV_{\check g_i})
 \xrightarrow{\mathrm{pmGH}}
 (\mathbb R^n,d_{\mathrm{Eucl}},0,dx).
\]
The zero mass indices and the finitely many initial indices are treated
as in the proof under harmonic asymptotics.

Each \(\widehat E_i\) is closed with smooth boundary in the complete
manifold \((M_i,\check g_i)\), so its intrinsic metric is complete.
Since
\[
 (1+\beta_i)^{-1}\check g_i\leq g_i
 \leq(1-\beta_i)^{-1}\check g_i,
\]
Lemma~\ref{lem:bilinear-transfer}, with
\(\vartheta_i:=\beta_i/(1-\beta_i)\), transfers the convergence to
\(g_i\) for the same base points.  It also gives
\[
 \mathcal H_{g_i}^{n-1}(\partial\widehat E_i)
 \leq(1-\beta_i)^{-(n-1)/2}
 \mathcal H_{\check g_i}^{n-1}(\partial\widehat E_i)
 \longrightarrow0.
\]
The theorem follows.
\end{proof}

\end{document}